\documentclass[11pt]{amsart}
\usepackage{amsmath,amssymb,amsthm, graphicx}
\usepackage{tikz, tikz-cd}
\usepackage{enumitem}
\usepackage{mathtools}
\usepackage{todonotes}

\usepackage{hyperref}

\newcommand{\Aut}{\mathop{\rm Aut}}
\newcommand{\MW}{\mathop{\rm MW}}
\newcommand{\MWL}{\mathop{\rm MWL}}
\newcommand{\NS}{\mathop{\rm NS}}

\newcommand{\Co}{\mathop{\rm Co}}

\newcommand{\FF}{{\mathbb F}}
\newcommand{\Het}[1]{H^{#1}_{\mathrm et}}

\newcommand{\II}{{\mathop{\rm II}}}
\newcommand{\PP}{{\mathbb P}}
\newcommand{\QQ}{{\mathbb Q}}
\newcommand{\CC}{{\mathbb C}}
\newcommand{\rank}{\mathop{\rm rk}}
\newcommand{\ZZ}{{\mathbb Z}}

\newcommand{\HM}[1]{\rm{no. {#1}}}
\newcommand{\lrbracket}[1]{\langle #1 \rangle}
\newcommand{\ddda}{\cdots_{(a)}}
\newcommand{\dddb}{\cdots_{(b)}}

\theoremstyle{plain}
\newtheorem{theorem}{Theorem}[section]
\newtheorem{proposition}[theorem]{Proposition}
\newtheorem{lemma}[theorem]{Lemma}
\newtheorem{corollary}[theorem]{Corollary}

\theoremstyle{definition}

\newtheorem{example}[theorem]{Example}
\newtheorem{remark}[theorem]{Remark}
\newtheorem{claim}[theorem]{Claim}
\newtheorem{conclusion}[theorem]{Conclusion}
\newtheorem{summary}[theorem]{Summary}
\begin{document}

\title[Symplectic automorphisms of K3 surfaces]{Finite symplectic automorphism groups of supersingular K3 surfaces}

\author{Hisanori Ohashi}
\address{Department of Mathematics, Faculty of Science and Technology,
Tokyo University of Science,
Noda 2641, Chiba, 278-8510, Japan}
\email{ohashi\underline{ }hisanori@rs.tus.ac.jp}

\author{Matthias Sch\"utt}
\address{Institut f\"ur Algebraische Geometrie, Leibniz Universit\"at
  Hannover, Welfengarten 1, 30167 Hannover, Germany}
  
      \address{Riemann Center for Geometry and Physics, Leibniz Universit\"at
  Hannover, Appelstrasse 2, 30167 Hannover, Germany}

\email{schuett@math.uni-hannover.de}

\date{April 28, 2026}

\begin{abstract}
We give a complete classification of finite groups acting
 symplectically on supersingular K3 surfaces of Artin invariant one.
 Using work of Dolgachev and Keum,
 this  provides the full classification of tame finite symplectic automorphism groups
 on any K3 surface, and in particular of all finite symplectic automorphism groups on K3 surfaces in characteristic $p>11$.
\end{abstract}

\thanks{The second author's research is partly conducted in the framework of the research training
group GRK 2965: From Geometry to Numbers,
funded by DFG}

\maketitle

\section{Introduction}

Automorphism groups of K3 surfaces have been a central area of research ever since Nikulin's 
seminal works \cite{Nikulin-finite}.
Focussing on symplectic automorphism, i.e. those leaving the regular 2-form invariant,
Nikulin classified all finite  abelian automorphism groups acting symplectically on complex K3 surfaces.
In \cite{Mukai}, this was extended to any symplectic finite group action over $\CC$ by Mukai.
In positive characteristic, however, the problem is still open,
except for a few wild group actions on supersingular K3 surfaces in small characteristic due to 
Kond\=o \cite{Kondo},
and important  results for the tame case by Dolgachev and Keum \cite{DKeum}
{ which greatly limit the possible groups beyond Mukai's classification, but leave existence open.}
Our first main result settles the existence questions and
thus offers the complete classification for tame group actions in any characteristic.
{ As a bonus, we develop a pure formulation} in 
abstract group theoretic terms in the spirit of Mukai's result,
involving the Legendre symbol.
A special role is taken by the supersingular K3 surface $X_{0,p}$ of Artin invariant $\sigma_0=1$
in characteristic $p$
(cf.\ the discussion around \eqref{eq:A}).

\begin{theorem}[tame case]
\label{thm}
A finite group
$G$ admits a tame symplectic action on some K3 surface $X$ in characteristic $p$
if and only if $p\nmid |G|$ and  $G$ can be realized as a subgroup of $M_{23}$  
%and without elements of order $>8$
whose action on $\{1,2,\hdots,24\}$ has
\begin{enumerate}
\item[(i)]
either at least 5 orbits
\item[(ii)]
or exactly 4 orbits 
such that the orbit lengths $l_1,\hdots,l_4$ satisfy the condition 
\begin{eqnarray}
\label{eq:Legendre}
\left(\dfrac{l_1\cdots l_4}p\right) = -1.
\end{eqnarray}
\end{enumerate}
{ 
Moreover, each action can be realized on $X_{0,p}$;
in case (ii), this is the only K3 surface realizing the $G$-action.}

\end{theorem}

The 10 maximal groups with 4 orbits were determined in \cite[Thm.\ 5.2]{DKeum} (see Table \ref{tab} which also lists the orbits lengths)
while the 11 maximal  groups with 5 orbits are reproduced  in \eqref{eq:max} from \cite{Mukai}.

{
To extend these results, we have to  take  wild symplectic group actions into account, 
where the characteristic $p$ divides the group order.
By \cite[Thm.\ 2.1]{DKeum}, this implies $p\leq 11$, whence Theorem \ref{thm} gives the complete classification
of all finite symplectic groups if $p>11$.
In the sequel, we only consider wild actions on supersingular K3 surfaces,
with special focus on $X_{0,p}$
to obtain the following  
result about the hierarchy of groups complementing Theorem \ref{thm}.}
Throughout we use  the notation from \cite{Atlas}, \cite{DKeum}.

%We let $X_{0,p}$ be the unique supersingular $K3$ surface with Artin invariant $\sigma_0 =1$ over 
%an algebraically closed field $k$ of odd characteristic $p$.

%Classification of the groups (in the notation of \cite{Atlas}):

%\begin{theorem}[tame case]
%\label{thm-tame}
%Let the finite group $G$ act tamely symplectically on a (supersingular K3 surface) in characteristic $p$.
%Then $p\nmid |G|$ and $G$ is contained in one of the following groups:
%
%\end{theorem}
%

\begin{theorem}[small characteristics]
\label{thm1}
A finite group $G$ acts  symplectically on 
%a supersingular K3 surface
%of Artin invariant $\sigma_0=1$ in characteristic $p\leq 11$
$X_{0,p}$ 
 in characteristic $p\leq 11$
if and only if
 %$p\leq 11$, $p\mid |G|$ and 
 $G$ is contained in one of the following groups:
\begin{enumerate}
\item[$(p=2)$]
$M_{21}.2_1$ of order $40,320 = 2^7.3^2.5.7$,\\
$\Aut(\mathfrak{S}_6)$
 of order $1,440 = 2^5.3^2.5$.
\item[$(p=3)$] $\mathrm{PSU}(4,\FF_3)$ of order $3,265,920=2^7.3^6.5.7$,\\
$2^2.\mathfrak A_{4,4}$ of order $1,152 = 2^7.3^2$,\\
$2^4:\mathfrak S_{3,3}$ of order $576 = 2^6.3^2$.
\item[$(p=5)$] $\mathrm{PSU}(3,\FF_5)$ of order $126,000=2^4.3^2.5^3.7$,\\ 
$\mathfrak{A}_8$ of order $20,160=2^6.3^2.5.7$,\\
 $2^4:(3\times \mathfrak{A}_5):2$ of order = $2^7.3^2.5$.
\item[$(p=7)$] $M_{21}.2_2$ of order $40,320 = 2^7.3^2.5.7$,\\
 $2^4:\mathfrak{A}_7$ of order $40,320 = 2^7.3^2.5.7$.
\item[$(p=11)$] $M_{22}$ of order $443,520 = 2^7.3^2.5.7.11$,
\\
$M_{11}$ of order
$7,920 = 2^4.3^3.5.11$,\\
$2^4:\mathfrak S_{3,3}$ of order $576 = 2^6.3^2$.
\end{enumerate}
\end{theorem}

{
Note that outside characteristic $p=2$,
Theorem \ref{thm1}
covers all subgroups of $M_{23}$ with 5 or more orbits,
regardless of being wild or not (cf.\ Corollary \ref{cor:anyM_23}).
}

The global characterization of the wild group actions follows immediately from Theorem \ref{thm1}
{ (in agreement with, but in fact independent of \cite[Thm.\ 2.1]{DKeum}):}

\begin{corollary}[wild case]
\label{cor}
A finite group $G$ acts wildly symplectically on 
$X_{0,p}$
%a supersingular K3 surface
%of Artin invariant $\sigma_0=1$ in characteristic $p$
if and only if
 $p\leq 11$, $p\mid |G|$ and $G$ is contained in one of the  groups from Theorem \ref{thm1}.
% (except for the last one in characteristic $11$ which is tame).
%\begin{enumerate}
%\item[$(p=2)$]
%$M_{21}.2_1$ of order $40,320 = 2^7.3^2.5.7$,\\
%$\Aut(\mathfrak{S}_6)$
% of order $1.440 = 2^5.3^2.5$.
%\item[$(p=3)$] $\mathrm{PSU}(4,\FF_3)$ of order $3,265,920=2^7.3^6.5.7$,\\
%$2^4:\mathfrak S_{3,3}$ of order $576 = 2^6.3^2$.
%\item[$(p=5)$] $\mathrm{PSU}(3,\FF_5)$ of order $126,000=2^4.3^2.5^3.7$,\\ 
%$\mathfrak{A}_8$ of order $20160=2^6.3^2.5.7$,\\
% $2^4(3\times \mathfrak{A}_5):2$ of order =$2^7.3^2.5$.
%\item[$(p=7)$] $M_{21}.2_2$ of order $40,320 = 2^7.3^2.5.7$,\\
% $2^4\mathfrak{A}_7$ of order $40,320 = 2^7.3^2.5.7$.
%\item[$(p=11)$] $M_{22}$ of order $443,520 = 2^7.3^2.5.7.11$
%and 
%\\
%$M_{11}$ of order
%$7,920 = 2^4.3^3.5.11$,\\
%$2^4:\mathfrak S_{3,3}$ of order $576 = 2^6.3^2$.
%\end{enumerate}
\end{corollary}

Combined with the results from \cite{DK-11} 
{ for the non-supersingular case   in characteristic $11$,}
we also deduce the following complete classification:

\begin{corollary}[characteristic $11$]
\label{cor:11}
A finite group $G$ acts symplectically on some K3 surface in characteristic $11$
if and only if $G$ is contained in $M_{11}, M_{22}$ or $2^4:\mathfrak S_{3,3}$.
\end{corollary}

We will employ a unified proof for both main theorems
which is largely independent of the arguments in \cite{DKeum};
{
concretely, the proof of Theorem \ref{thm} relies on  \cite{DKeum} to include the case of finite height,
while Theorem \ref{thm1} with its focus on $X_{0,p}$ is completely independent,
and so is Corollary \ref{cor} (cf.\ Remark \ref{rem:p11}).
}
In short, we will use work of Nygaard to prove in Theorem \ref{discri-action} 
that the induced action on the discriminant group is trivial
(just like in the complex case).
This allows us to employ lattice theoretic arguments 
(in particular the crystalline Torelli theorem in characteristic $p>3$, see Theorem \ref{thm:cTorelli}).
%and reduce to the supersingular K3 surface of Artin invariant one (Proposition \ref{prop:reduction}).
Then we translate the problem into symmetries of the Leech lattice (Theorem \ref{thm:Leech})
and make great use of the classification by H\"ohn--Mason \cite{HM290}
which reduces the classification of the possible groups to discriminant group computations
{ as laid out in Section \ref{s:plan}.
}
For this purpose, we also prove a useful analogue of the Witt cancellation theorem (Lemma \ref{lem:Witt}).
Finally existence is proven using the crystalline Torelli theorem or, when this is not available,
by exhibiting explicit models in characteristics $2, 3$
admitting the required symplectic group actions (Section \ref{s:models}).

We end the introduction with two interesting corollaries.
The first concerns cohomology representations of finite symplectic group actions
in the spirit of \cite{Mukai}:

\begin{theorem}
\label{thm:cohom}
Let $G$ be a finite  group acting symplectically on 
$X_{0,p}$.
%some supersingular K3 surface $X$
%of Artin invariant $\sigma_0=1$. 
Then, { at least} one of the following holds:
\begin{enumerate}
\item $G\subset M_{23}$ decomposes $\Omega$ into at least 3 orbits
and $\Het{*}(X, \QQ_\ell)$ admits a Mathieu representation of $G$ ($\ell\neq p$) 
\item
% NOTE: the action of G may not be Mathieu, see some other appearance of S_4 etc, but there is such a
$p=2$, $G$ contains a wild element and $G\subset M_{21}.2_1$ or $G\subset\Aut(\mathfrak S_6)$;
\item $p=3$, $G$ contains a wild element and $G\subset \mathrm{PSU}(4,\FF_3)$;
\item $p=5$, $G$ contains a wild element and $G\subset \mathrm{PSU}(3,\FF_5)$.
\end{enumerate}
\end{theorem}

The second corollary determines when there are no extra groups acting symplectically in characteristic $p$,
i.e.\ only case (i) in Theorem \ref{thm} occurs:

\begin{theorem}
\label{thm:holds}
Mukai's classification holds true in characteristic $p$ if and only if $2, 3, 5, $ and $7$ are all non-zero squares modulo $p$.
Equivalently, modulo $840$, we have that $p\in\{\pm 1^2, \pm 11^2, \pm 13^2, \pm 17^2, \pm 19^2, \pm 23^2\}$.
\end{theorem}

\begin{remark}
After posting this paper, there appeared the preprint \cite{WZ} on arXiv.
Notably, this gives an independent proof of some of the results for tame group actions from \cite{DKeum}
which enter in the proof of Theorem \ref{thm}.
It also develops a new technique (called $p$-root pairs) to study wild group actions;
possibly this could be useful in extending the results of the present paper to supersingular K3 surfaces of higher Artin invariant.
\end{remark}

\subsection*{Notation and conventions}

All group actions are assumed to be faithful.

Root lattices as well as the Leech lattice are taken to be negative-definite
(contrary to the convention in \cite{HM290}, for instance).

We use the genus symbols defined by Conway-Sloane \cite[Chapter 15]{CS99}, including the refinement from \cite{AGM}, 
and the gluing theory of 
Kneser and Nikulin \cite{Nikulin80} almost interchangeably. 

For groups, we follow \cite{Atlas, DK, HM290, Mukai} to denote

\smallskip

$A:B$ --  semi-direct product of subgroups $A, B$ where $A$ is normal

\smallskip

$A.B$ -- a group $G$ with normal subgroup $A$ such that $G/A\cong B$

\smallskip

$\mathfrak S_{m,n} \cong \mathfrak S_m\times\mathfrak S_n$

\smallskip

$\mathfrak A_{m,n} = \mathfrak S_{m,n} \cap \mathfrak A_{m+n}$

\smallskip

$n^e \cong (\ZZ/n\ZZ)^e$

%But, for $p=2$, we still avoid the use of 
%Conway-Sloane's extension of Legendre symbol to denominator 2 and make formulas explicit as possible.

\section{Symplectic automorphisms on K3 surfaces}
\label{s:basics}

Let $X$ be a  K3 surface,
i.e.\ a smooth projective surface over an algebraically closed field $k$ of characteristic $p\geq 0$
with zero irregularity and trivial canonical bundle:
\[
h^1(X, \mathcal O_X)=0, \;\;\; \omega_X\cong\mathcal O_X.
\]
An automorphism $g$ of $X$ is called  symplectic
if it leaves the regular 2-form invariant.
In \cite{Nikulin-finite}
Nikulin proved the following result for complex K3 surfaces, algebraic or analytic:

\begin{theorem}[Nikulin]
\label{thm:Nikulin}
A non-trivial symplectic automorphism $g$ of a complex K3 surface has only isolated fixed points.
Their number depends only on the order of $g$ as follows:
%\begin{table}[ht!]
$$
\begin{array}{|c|ccccccc|}
\hline
n & 2 & 3 & 4 & 5 & 6 & 7 & 8\\
\hline
\#\mathrm{Fix}(g) & 8 & 6 & 4 & 4 & 2 & 3 & 2\\
\hline
\end{array}
$$
%\end{table}
\end{theorem}
By \cite{DKeum}, the same classification holds true for tame symplectic automorphisms 
of K3 surfaces in positive characteristic.

In \cite{Mukai},
Mukai exposed a striking connection with the Mathieu group $M_{23}$, acting on the set of 24 letters.
Namely the above numbers agree exactly with the number of fixed points of any element $\sigma\in M_{23}$ 
of the given order:

\begin{table}[ht!]
$$
\begin{array}{|c|ccccccccccc|}
\hline
n & 2 & 3 & 4 & 5 & 6 & 7 & 8 & 11 & 14 & 15 & 23\\
\hline
\#\mathrm{Fix}(\sigma) & 8 & 6 & 4 & 4 & 2 & 3 & 2 & 2 & 1 & 1 & 1\\
\hline
\end{array}
$$
\caption{Fixed points of order $n$ elements of $M_{23}$}
\label{tab:23}
\caption{tab:orders}
\end{table}

\vspace{-.4cm}

For a complex K3 surface $X$,
any symplectic automorphism acts trivially on the transcendental lattice,
the smallest sub-Hodge structure of $H^2(X,\ZZ)$
whose complexification contains the (invariant) 2-form.
Since this has  rank at least 2 (due to  the Hodge decomposition),
the invariant sublattice 
\[
H^*(X,\ZZ)^G
\]
will have rank 5 at least for any finite symplectic group $G$ of automorphisms,
additional invariants coming from $H^0(X,\ZZ), H^4(X,\ZZ)$ and an ample or K\"ahler class.
This can be seen as the origin of the 5 orbits in Mukai's  theorem (and also in Theorem \ref{thm} (i);
compare also Theorem \ref{thm:equiv}):

\begin{theorem}[Mukai]
\label{thm:Mukai}
A finite group $G$ admits a symplectic action on a complex K3 surface
if and only if $G$ is a subgroup of $M_{23}$ with at least 5 orbits.
\end{theorem}

%Mukai also gives examples for each maximal group (listed in \eqref{eq:max}),
%and one easily verifies that these have good reduction when the characteristic\
%does not divide the group order (Lemma \ref{lem:Mukai}).
%This accounts for most of part (i) of Theorem \ref{thm},
%see Section \ref{ss:pf} for the details.

By \cite[\S 4]{DKeum}, Mukai's techniques  carry over to
tame actions on K3 surfaces  in positive characteristic
(using $\ell$-adic \'etale cohomology).
For K3 surfaces of finite height,
or equivalently, by the Tate conjecture (\cite{Charles}, \cite{Maulik}, \cite{MP15}, \cite{KMP16}),
with Picard number $\rho(X)\leq 20$,
this allows to recover { the same classification of possible symplectic groups as over $\CC$
\cite[Theorem 4.7]{DKeum}.
The existence of these group actions on K3 surfaces of finite height 
was never addressed explicitly to the best of our knowledge;
in this paper,
we will (only) show that all these groups are realized on the supersingular K3 surface $X_{0,p}$ 
(outside characteristic $2$, cf.\ Theorem \ref{thm:equiv}).
For the finite height case, 
one could approach the problem by constructing a suitable singular K3 surface admitting the required group action 
(using the classification of co-invariant lattices in \cite{Hashi} and choosing an appropriate polarization) 
and specializing to an ordinary K3 surface, possibly after a base extension following \cite[Cor.\ 0.5]{Matsumoto}.}
In order to prove Theorem \ref{thm} and the other  results of this paper, we will  focus on supersingular K3 surfaces in the sequel.

\section{Torelli theorems for supersingular $K3$ surfaces}\label{torelli}

%Let $X$ be a $K3$ surface over an algebraically closed field $k$ of positive characteristic $p>0$.
%A subgroup $G\subset \Aut (X)$ is said to be {\em{symplectic}} if it acts on the space $H^0(X, \omega_X)$ of global 2-forms trivially. We note that (smooth) $K3$ surfaces do not admit infinitesimal 
%automorphisms by the result of Rudakov-Shafarevich \cite{RS76, Nygaard79}.
A K3 surface $X$ is called
%We say that $X$ is
 {\em{supersingular}} if its Picard number $\rho(X)$ equals 22, the maximum possible value. % (Shioda-supersingular). 
 Equivalently, by the Tate conjecture, 
 %which was completed recently by many authors (see \cite{MP15, KMP16}), %\memo{Madapusi-Pera (p>=3), Madapusi-Pera--Kim(p=2)},
%this is equivalent to saying that 
the formal Brauer group of $X$ is unipotent. % (Artin-supersingular). 
We first state a characterization of symplectic automorphisms
in terms of the discriminant group
\begin{eqnarray}
\label{eq:A}
A(\NS(X))=\NS(X)^\vee/\NS(X)
\end{eqnarray}
of the N\'{e}ron-Severi lattice $\NS(X)$. 
By \cite{Artin}, this is a $p$-elementary abelian group of size
$2\sigma_0$
where $\sigma_0\in\{1,\hdots,10\}$ is called the \emph{Artin invariant} of $X$.

The following result may be known to experts, at least for $p>2$
(for example, compare \cite{Jang19}),
but we will crucially also need it for $p=2$.

\begin{theorem}
\label{discri-action}
Let $X$ be a supersingular $K3$ surface. Then,
an automorphism $g\in \Aut (X)$ is symplectic if and only if $g$ acts trivially on %the discriminant group
$A(\NS(X))$.
%=\NS(X)^*/\NS(X)$ of the N\'{e}ron-Severi lattice $\NS(X)$. 
\end{theorem}

\begin{proof}
%This is a corollary to 
We first consider the case $p>2$.
Then
Nygaard \cite[Theorem 2.1]{Nygaard80} shows that 
for any endomorphism $g$ of $X$ there exists a $\lambda_0\in k$ such that 
the action of $g$ on the Witt vector cohomology $H^2(W\mathcal{O}_X)$
is given by the formula
\[\sum_{i=0}^{\infty} b_iV^i \mapsto \sum_{i=0}^{\infty} b_i \lambda_0^{1/p^i} V^i\]
under the isomorphism $H^2(W\mathcal{O}_X)\simeq k_{\sigma}[[V]]$ as covariant Dieudonn\'{e} modules \cite{AM77}.
From the exact sequence 
\[0 \rightarrow H^2(W\mathcal{O}_X) \stackrel{V}{\rightarrow} H^2(W\mathcal{O}_X) \rightarrow H^2(\mathcal{O}_X) \rightarrow 0,\]
the action of $g$ on $H^2(\mathcal{O}_X)$ is given by $\lambda_0$. Therefore, by
Serre duality, we have
\[\begin{split}
g \text{ is symplectic }&\Leftrightarrow  g\text{ acts trivially on } H^2(\mathcal{O}_X)\\
&\Leftrightarrow g \text{ acts trivially on }H^2(W\mathcal{O}_X).
\end{split}\]
On the other hand, \cite[Theorem 1.12]{Nygaard80} gives us a canonical isomorphism 
\[\ker F^{\sigma_0}d\simeq A(\NS(X))\otimes k,\]
where $d=d_1^{0,2}\colon H^2(W\mathcal{O}_X)\rightarrow H^2(W\Omega^1_X)$ 
is the differential in the slope spectral sequence and $\sigma_0$ is the Artin invariant of $X$.
This space has a natural basis $\{1,V,\dots,V^{2\sigma_0-1}\}$ comprising eigenvectors of $g$
by the formula above. 
Hence, $g$ acts trivially on $H^2(W\mathcal{O}_X)$ if and only if it acts trivially on the discriminant group
$A(\NS(X))\subset A(\NS(X))\otimes k$.

In characteristic $p=2$, some extra care is necessary because \cite[Theorem 1.12]{Nygaard80} and its reference
\cite[Section 3]{Ogus79} are stated only for odd 
characteristics. In fact, the argument above goes through with a small modification as follows. 
The decomposition
\[(H^2_{fl}(X, \mathbb{Z}_p(1)), (\ ,\ ))\simeq (T_0,p\langle\ ,\ \rangle)\oplus (T_1, (\ ,\ ))\]
of $p$-adic quadratic forms as in \cite[(3.13.3)]{Ogus79} for $p=2$ can be obtained instead by a 
general decomposition procedure for $2$-adic lattices explained in \cite[Chapter 15]{CS99}. 
%the explicit forms of N\'{e}ron-Severi lattices in $p=2$ in \cite{RS81}. 
Using this, the proof of \cite[Theorem 1.12]{Nygaard80} remains valid and the proof is done.
\end{proof}
%We remark that the theorem above answers a question in \cite{Kondo}: The six group actions on supersingular
%$K3$ surfaces constructed by Kondo \cite{Kondo} are all symplectic by the construction.

\subsection{Finite group actions}

It is known that the representation $\Aut (X) \rightarrow O(\NS(X))$ is faithful for 
supersingular $K3$ surfaces in positive characteristic $p\geq 2$ \cite[Section 8]{RS81}.
For a finite group $G$ of isometries of a lattice $L$, the
invariant and coinvariant sublattices are of importance: 
\begin{equation*}
\begin{split}
L^G&=\{l\in L\mid g(l)=l\ \forall g\in G\},\\
L_G& %=\text{(the orthogonal complement of $L^G$ in $L$)}
= (L^G)^\perp\subset L.
\end{split}
\end{equation*}
We state the crystalline Torelli theorem in an adapted form.

\begin{theorem}\label{thm:cTorelli}
Let $N$ be a lattice which is isomorphic to the N\'{e}ron-Severi lattice of a supersingular $K3$ surface 
of Artin invariant $\sigma_0$ in characteristic $p> 3$. 
Suppose we have a finite group $G$ acting 
faithfully on $N$ with the following properties:
\begin{enumerate}
\item[(1)]
the coinvariant lattice $N_G$ is negative definite and contains no vectors of norm $(-2)$.
\item[(2)]
 $G$ acts trivially on the discriminant group $N^*/N$. 
 \end{enumerate}
Then, there exists a supersingular $K3$ surface $X$ and an isometry $N\simeq \NS(X)$
such that the induced action of $G$ on $\NS(X)$ is realized geometrically as a group of 
symplectic automorphisms of $X$. 
\end{theorem}

\begin{proof}
We choose a positive cone $\mathcal{C}\subset N\otimes \mathbb{R}$
%, one of the two cones of positive norm vectors, 
and consider the chamber (i.e., connected component) decomposition 
$\mathcal{C}\setminus\cup H_{\delta}$, where $H_{\delta}$ is the hyperplane in $N\otimes \mathbb{R}$ 
orthogonal to $\delta$ which runs over elements $\delta\in N$ of norm $(-2)$,
the so-called roots.
By condition (1), some chamber $\Delta$ meets $N^G\otimes \mathbb{R}$. 
Then, with an {\em{arbitrary}} choice of a strictly characteristic subspace $K\subset A(N)\otimes k$, 
we can find a supersingular $K3$ surface $X$ with marking $N\simeq NS(X)$ such that $\Delta$
corresponds to the ample cone of $X$ by Ogus' crystalline Torelli theorem \cite[Theorem III]{Ogus83}.
The realization statement now follows from \cite[Theorem II]{Ogus83} since $G$ and the crystalline first chern 
class map automatically commute by the condition (2) and \cite[Proposition 1.9]{Ogus83}.
%Note that for $p=3$, the crystalline Torelli theorem is not direct from \cite{Ogus83} (see the footnote in p.364), 
%but the required potentially good reduction statement for supersingular $K3$ surfaces in $p=3$ is 
%proved in \cite{BL19}, so we have the same Torelli formulations.
Finally, the automorphisms are symplectic by Theorem \ref{discri-action}.
\end{proof}

\begin{remark}
\label{rem:crys}
For $p=3$, the crystalline Torelli theorem is not proved in \cite{Ogus83} (see the footnote on p.\ 364), 
but the required potentially good reduction statement for supersingular $K3$ surfaces in $p=3$ is 
stated in \cite{BL19}.
However, we will presently not need this
%\marginpar{really?}
since we can prove existence
by exhibiting explicit models with symplectic $G$-action in characteristics $p=2,3$
(see Section \ref{s:models}).
\end{remark}

In fact, the finiteness of $G$ is not used in the proof above.
In particular, since the strictly characteristic subspaces can be chosen arbitrary, we have the following.

\begin{corollary}
The subgroup $\mathrm{Aut}^{\text{symp}}(X)$ of symplectic automorphisms of a supersingular $K3$ surface $X$
depends only on the characteristic $p> 3$ and the Artin invariant $\sigma_0$, but not on $X$. This group is 
isomorphic to $O^+(\NS(X))/W_X$,
where $O^+(\NS(X))$ is the subgroup of isometries which preserves the positive cone $\mathcal{C}$ and acts
trivially on the discriminant group $A(\NS(X))$, and $W_X$ is the Weyl group generated by $(-2)$-reflections.
\end{corollary}

This corollary has an obvious analogue for complex $K3$ surfaces,  the proof of which is easy.

\section{The supersingular $K3$ surface of Artin invariant 1}

It is known that the supersingular $K3$ surface with Artin invariant $\sigma_0=1$ is 
unique up to isomorphisms (see \cite{Ogus79} for $p\geq 3$ %Corollary 7.14
and \cite{RS,DK} for $p=2$; { it is sometimes called superspecial}).
We shall denote it by $X_{0,p}$ or, if $p$ is clear, by $X_0$.
%Proposition \ref{prop:reduction} thus reduces the proofs of Theorems \ref{thm} and \ref{thm1}
%to a study of $X_0$.
%(which was also a central object of study in \cite{DKeum}.)

%Let $X$ be a supersingular $K3$ surface with Artin invariant $\sigma_0$ over an algebraically closed field $k$ of characteristic $p$. 
%It is known that the supersingular $K3$ surface with Artin invariant $\sigma_0=1$ is 
%unique up to isomorphisms (see \cite{Ogus79} for $p\geq 3$ %Corollary 7.14
%and \cite{RS,DK} for $p=2$).
%In this section, we focus on this case $\sigma_0=1$ and denote this unique surface by $X_0$.
The N\'{e}ron-Severi lattice $\NS(X_0)$ is even, of signature $(1,21)$ and $p$-elemen\-tary with 
discriminant $-p^2$. These information amount to saying that the lattice $\NS(X_0)$ belongs to the genus 
\[\begin{split}
\II_{1,21}\, &p^{\varepsilon 2},\quad \varepsilon=-\left(\frac{-1}{p}\right)\ \ (p\geq 3),\\
\II_{1,21}\, &2_{\II}^{-2}\ \ (p=2).
\end{split}\]
Here, we use  Conway-Sloane's genus symbol for which we refer to \cite[chapter 15]{CS99} and the 
bracket indicates the Legendre symbol. The isomorphism class of the lattice 
is unique in its genus by \cite{RS81,CS99}.
We shall now develop an instrumental relation with the Leech lattice.

\subsection{Relation with the Leech lattice}
\label{RtoL}

{
Using recent work of Marquand and Muller \cite{MM}, we shall develop a crucial relation 
between the symplectic $G$-action on $X_0$ and a suitable counterpart on the Leech lattice $\Lambda$.
For previous approaches to this problem, see the discussion after Proposition \ref{Leech-type}.
%in \ref{ss:Borcherds}.

\begin{theorem}
\label{cor:rk<=3}
\label{thm:Leech}
Every finite symplectic subgroup $G\subset \Aut (X_0)$ can be seen as a subgroup of the Conway group 
$\Co_0=O(\Lambda)$ such that
$\Lambda_G\cong\NS(X_0)_G$ and $\rank(\Lambda^G)\geq 3$.
\end{theorem}

\begin{proof}
Note that $G$
acts faithfully on $\NS(X_0)$ (\cite{RS81}).
We know that $\NS(X_0)_G$ is negative definite of rank at most $21$,
since $G$ fixes some ample class on $X_0$.
Moreover, by definition, $G$ leaves no nonzero vector in $\NS(X_0)_G$ invariant,
and $G$ acts trivially on the discriminant group $A({\NS(X_0)_G})$ by Theorem \ref{discri-action}.
Hence $\NS(X_0)_G$ is stable symplectic in the terminology of \cite[\S 4.1]{MM}.

We want to apply \cite[Cor 4.19]{MM} deduce that $\NS(X_0)_G$ embeds primitively into $\Lambda$.
To this end, we have to check that $\NS(X_0)_G$ contains no roots (by construction, cf.\ Theorem \ref{thm:cTorelli})
and that
\begin{eqnarray}
\label{eq:<=24}
\rank(\NS(X_0)_G) + l(\NS(X_0)_G) \leq 24.
\end{eqnarray}
Here $ l(\NS(X_0)_G)$ denotes the length of the discriminant group $A(\NS(X_0)_G)$,
i.e.\ the minimum number of generators,
and the bound in \eqref{eq:<=24} can be seen as follows:
Since $\NS(X_0)_G$ embeds primitively into $\NS(X_0)$, we have
\[
l(\NS(X_0)_G) \leq \underbrace{l(\NS(X_0))}_{=2} + \underbrace{l(\NS(X_0)^G)}_{\leq \rank(\NS(X_0)^G)}
\leq 24-\rank(\NS(X_0)_G),
\]
giving directly the claim.
Now, with the primitive embedding 
$\NS(X_0)_G\hookrightarrow\Lambda$,
we can extend the $G$-action on $\NS(X_0)_G$ by the identity on $(\NS(X_0)_G)^\perp\subset\Lambda$
to a $G$-action on the whole of $\Lambda$ (because $G$ acts trivially on $A({\NS(X_0)_G})$).
Then $\Lambda_G\cong\NS(X_0)_G$ follows by construction, and $\rank(\Lambda^G)=24-\rank(\Lambda_G)\geq 3$ is obvious.
\end{proof}

{
\begin{remark}
\label{rem:p11}
The corollary  offers an independent proof that wild automorphisms on $X_{0,p}$ only occur in characteristic $p\leq 11$,
in agreement with \cite{DK-11}.
Indeed, by \cite{HM290}, all prime divisors of the orders of subgroups of  
$\Co_0=O(\Lambda)$ with  $\rank(\Lambda^G)\geq 3$ satisfy $p\leq 11$.
\end{remark}
}

\subsection{Connection through the Borcherds lattice}
\label{ss:Borcherds}

For a better intuition, we interpret the above results through a joint connection with the Borcherds lattice $L_{1,25}$,
the unique even unimodular lattice 
%$L\in \II_{1,25}$
of signature $(1,25)$.
}

For $p\geq 3$, let us pick up a negative definite lattice $H^{(p)}$ from the genus $\II_{0,4}\, p^{\varepsilon 2}$.
We note that the non-emptiness of this genus follows from 
\cite[chapter 15, Theorem 11]{CS99}, and an explicit form 
of $H^{(p)}$ is given in \cite{Shimada04} for example. %We easily see that $\NS(X_0)\oplus H^{(p)}$ is in the genus $\II_{1,25}\ p^{4}$. 
For $p=2$, we simply put $H^{(2)}=D_4$ for the negative definite root lattice of type $D_4$, whose genus is indeed
$\II_{0,4}\,2_{\II}^{-2}$. 
The next lemma is a special case of the Kneser-Nikulin's gluing method \cite[Lemma 1.6.1]{Nikulin80}.

\begin{lemma}\label{overlattice}
The orthogonal sum $\NS(X_0)\oplus H^{(p)}$ can be embedded into 
%an even unimodular lattice 
%$L\in \II_{1,25}$
$L_{1,25}$ such that $\NS(X_0)$ and $H^{(p)}$ are both primitive.
\end{lemma}

For later use, we record the genus symbol of the negated lattice.

\begin{lemma}\label{-1}
Let $p$ be a prime and $L$ an integral lattice whose $p$-adic symbol is of the form 
\[
%&1^{\varepsilon_0 n_0}p^{\varepsilon_1 n_1}\cdots (p^e)^{\varepsilon_e n_e}\cdots \quad (p\geq 3),\\
1_{(t_0)}^{\varepsilon_0 n_0}p_{(t_1)}^{\varepsilon_1 n_1}\cdots (p^e)_{(t_e)}^{\varepsilon_e n_e}\cdots, \]
where $t_i\in \ZZ/8\ZZ\cup \{\II\}$ appears only when $p=2$.
Then, the $p$-adic symbol of the negated lattice $L(-1)$ is as follows,
\[1_{(t'_0)}^{\varepsilon'_0 n_0}p_{(t'_1)}^{\varepsilon'_1 n_1}\cdots (p^e)_{(t'_e)}^{\varepsilon'_e n_e}\cdots, \]
where $\varepsilon'_i=\left(\frac{-1}{p}\right)^{n_i}\varepsilon_i\ (p\geq 3)$ and 
$\varepsilon'_i=\varepsilon_i,\ t'_i=-t_i\ (t_i\in \ZZ/8\ZZ),\ t'_i=\II\ (t_i=\II)\ (p=2)$.
\end{lemma}
The proof of this lemma is immediate from the definitions of $p$-adic symbols.
An analogous statement holds for the orthogonal complement of $L$ in any unimodular (or $p$-adically unimodular) lattice $L_0$
if $p>2$,
with the unimodular part of the $p$-adic symbol adjusted to the rank and determinant of $L_0$.
In case $p=2$, however, one has to take into consideration the subtleties of oddity fusion and sign walking
{ which prevent a $2$-adic symbol from being determined uniquely by a lattice
 (see \cite[p.\ 381]{CS99}).}

Next, suppose that $G\subset \Aut(X_0)$ is a finite group of symplectic automorphisms acting 
on $X_0$
(acting trivially on the discriminant group $A(\NS(X_0))$ by Theorem \ref{discri-action}).
{ Consider  the sign-reversing isometry 
\[
\delta \colon A(\NS(X_0))\rightarrow 
A(H^{(p)})
\]
which encodes how $L_{1,25}$ arises as an overlattice of $\NS(X_0)\oplus H^{(p)}$ in Lemma \ref{overlattice}.}
Then, regarding $H^{(p)}$ as a lattice with the trivial $G$-action, 
the graph 
of $\delta$ %(in the proof of Lemma \ref{overlattice}) 
is preserved under $G$, whence
the lattice $L_{1,25}$ admits an extended action by $G$ as before. 
Using the notation $L^G$ and $L_G$ for the (co-)invariant lattices as in Section \ref{torelli}, 
we have the following properties.

\begin{proposition}\label{Leech-type}\ \\
\textrm{(1)} $H^{(p)}$ is contained in $L_{1,25}^G$ and $(L_{1,25})_G$ is contained in $\NS(X_0)$. \\
\textrm{(2)} $L_{1,25}^G$ contains an element of strictly positive norm.\\
\textrm{(3)} $(L_{1,25})_G$ is negative definite with no elements of norm $(-2)$. \\
\end{proposition}

\begin{proof}
(1) is obvious from the construction. The element in (2) is the orbit sum of any ample divisor on $X_0$. 
(3) 
{ is a consequence of the construction 
as $(L_{1,25})_G\cong\NS(X_0)_G$, 
cf.\ the discussion in the proof of Theorem \ref{thm:Leech}.}
%Riemann-Roch theorem together with the fact that $G$ preserves the positivity on $X_0$.
\end{proof}

{
Now Theorem \ref{thm:Leech} implies that the $G$-action on $L_{1,25}$ can be taken
to fix a hyperbolic plane $U\subset (L_{1,25})^G$,
such that $G$ acts effectively on $U^\perp\cong\Lambda$.
Our first attempt to show this was a  variant of the approach in \cite[Proposition 2.2.]{Huybrechts16}, 
whose idea originates in the paper \cite[Appendix B]{GHV12}.
 Note, however, that there was a gap in the original arguments
which was pointed out and  fixed to the extent needed for our (and all other previous) considerations in \cite[Corollary 4.19]{MM}
(cf.\ also \cite{Zheng} for a conceptual extension).}

The relations between sublattices of $L=L_{1,25}$ can be depicted as follows.

\begin{figure}[ht!]
\begin{center}
\begin{picture}(200,90)
%%%%%%%

\linethickness{0.7pt}
\put(0,80){\line(1,0){200}}
\put(0,50){\line(1,0){200}}
\put(0,20){\line(1,0){200}}

\put(-15,76){$L$}
\put(-40,46){$G\curvearrowright L$}
\put(-36,16){\text{Leech}}

\linethickness{0.5pt}
\multiput(0,78)(0,-6){10}{\line(0,-1){3}}
\multiput(100,78)(0,-6){10}{\line(0,-1){3}}
\multiput(140,78)(0,-6){6}{\line(0,-1){3}}
\multiput(125,53)(0,-6){6}{\line(0,-1){3}}
\multiput(200,83)(0,-6){11}{\line(0,-1){3}}
%%%%%%%%%%

\linethickness{1.2pt}

\put(0,85){\line(0,-1){10}}
\put(140,85){\line(0,-1){10}}
\put(200,85){\line(0,-1){10}}
\put(60,83){$\NS(X_0)$}
\put(160,83){$H^{(p)}$}
\put(0,55){\line(0,-1){10}}
\put(100,55){\line(0,-1){10}}
\put(200,55){\line(0,-1){10}}
\put(40,53){$L_G$}
\put(146,53){$L^G$}
\put(0,25){\line(0,-1){10}}
\put(100,25){\line(0,-1){10}}
\put(125,25){\line(0,-1){10}}
\put(200,25){\line(0,-1){10}}
\put(40,23){$\Lambda_G$}
\put(155,23){$\Lambda^G$}
\put(109,23){$U$}

\end{picture}
\vspace{-.5cm}
\caption{Primitive sublattices of $L=L_{1,25}$}\label{subofL}
\end{center}
\label{fig:1,25}
\end{figure}

%\begin{figure}[ht!]
%\begin{center}
%\begin{picture}(200,90)
%%%%%%%%
%\linethickness{0.7pt}
%\put(0,80){\line(1,0){200}}
%\put(0,50){\line(1,0){200}}
%\put(0,20){\line(1,0){200}}
%\put(-15,76){$L$}
%\put(-40,46){$G\curvearrowright L$}
%\put(-36,16){\text{Leech}}
%\linethickness{0.5pt}
%\multiput(0,83)(0,-6){11}{\line(0,-1){3}}
%\multiput(100,83)(0,-6){11}{\line(0,-1){3}}
%\multiput(140,83)(0,-6){7}{\line(0,-1){3}}
%\multiput(170,53)(0,-6){6}{\line(0,-1){3}}
%\multiput(200,83)(0,-6){11}{\line(0,-1){3}}
%%%%%%%%%%%
%\linethickness{1.2pt}
%\put(0,85){\line(0,-1){10}}
%\put(140,85){\line(0,-1){10}}
%\put(200,85){\line(0,-1){10}}
%\put(60,83){$\NS(X_0)$}
%\put(160,83){$H^{(p)}$}
%\put(0,55){\line(0,-1){10}}
%\put(100,55){\line(0,-1){10}}
%\put(200,55){\line(0,-1){10}}
%\put(40,53){$L_G$}
%\put(144,53){$L^G$}
%\put(0,25){\line(0,-1){10}}
%\put(100,25){\line(0,-1){10}}
%\put(170,25){\line(0,-1){10}}
%\put(200,25){\line(0,-1){10}}
%\put(40,23){$\Lambda_G$}
%\put(130,23){$\Lambda^G$}
%\put(182,23){$U$}
%\end{picture}
%\caption{Primitive sublattices of $L=L_{1,25}$}\label{subofL}
%\end{center}
%\end{figure}
%

From Figure \ref{subofL}, we see that the primitive hull of $\Lambda_G\oplus H^{(p)}$ in $L$
is the orthogonal complement of $\NS(X_0)^G$ in $L$. 
By \cite[Prop.\ 1.6.1]{Nikulin80}, there exists a sign-reversing isometry $\gamma$ between subgroups of the discriminant groups $A(\Lambda_G)$ and $A(H^{(p)})$ such that the graph $\Gamma$ of 
$\gamma$ is a totally isotropic subgroup of $A(\Lambda_G)\oplus A(H^{(p)})$ and 
$\Gamma^{\perp}/\Gamma$ is isomorphic to the discriminant group $A(\NS(X_0)^G)$ with negated quadratic form. 
Our framework is that, under the knowledge of $A(\Lambda_G)$ (by \cite{HM290}) and $A(H^{(p)})$, 
we can look for possible $\gamma$ with the computation of $\Gamma^{\perp}/\Gamma$, thus reducing 
the problem to finite abelian groups with quadratic forms. However, since we know the ranks and signatures 
of $\Lambda_G$ and $H^{(p)}$, \cite[Corollary 1.9.4]{Nikulin80} shows in turn that the finite quadratic form
$\Gamma^{\perp}/\Gamma$ determines the genus of the even lattice $(\NS(X_0)^G)^{\perp}$. Therefore, 
we will indicate this change of genera under $\Lambda_G\oplus H^{(p)}\subset (\NS(X_0)^G)^{\perp}$ by the symbol $\rightsquigarrow$ using the Conway-Sloane's genus symbol, as seen many times in the following 
subsections.

 We conclude this section with some useful properties related to discriminant groups.
For a prime $p$ and a nondegenerate integral lattice $L$, we define the $p$-length 
using the $p$-Sylow subgroup $A(L)_p$ of the discriminant group $A(L)$:
\[
l_p(L)=\text{(the number of minimal generators of $A(L)_p$).}
\]
This is easily read off from the $p$-adic genus symbol of $L$ as follows. If the symbol is as in Lemma \ref{-1},
then \[l_p(L)=n_1+n_2+\cdots+n_e+\cdots.\]
It is clear then that $l_p(L)\leq \mathrm{rank} (L)$. 

\begin{corollary}\label{length-condition}
Let $p$ be the characteristic of the base field of $X_0$.\\
(1) The graph $\Gamma\subset A(\Lambda_G)\oplus A(H^{(p)})$  of $\gamma$
above
is $p$-elementary subgroup of order either 
$1$, $p$ or $p^2$.\\
(2) For primes $q$ different from $p$, the $q$-Sylow subgroup $A(\NS(X_0)^G)_q$ 
coincides with the negative of the form  
$A(\Lambda_G)_q$. In particular, we have $l_q(\Lambda_G)\leq \mathrm{rank} (\NS(X_0)^G)  = \mathrm{rank}(\Lambda^G)-2.$
\end{corollary}

\begin{proof}
Statement (1) follows since the abelian group $A(H^{(p)})$ is $p$-elementary of order $p^2$ and $H^{(p)}$ 
is a primitive sublattice of $L_{1,25}$. 
Since the Sylow subgroups at different primes are orthogonal to each other
(Jordan decomposition), (2) follows by the same reasoning.
\end{proof}

\section{Outline of the  proofs of the main theorems}
\label{s:plan}

In what follows, we will consider the groups $G$ acting on $\Lambda$ with $\Lambda^G$ of rank three at least
(by Theorem \ref{cor:rk<=3}).
We use the classification of \cite{HM290} to check whether $\Lambda_G$ embeds into $\NS(X_{0,p})$ at some suitable prime,
though sometimes, especially if $p=2$, 
it will be more straight forward to check whether $\Lambda_G\oplus H^{(p)}$ embeds into $L_{1,25}$.
If it does not embed for some $p$, then the given $G$-action on $\Lambda$ can be excluded
(though $G$ may admit another admissible action on $\Lambda$).
If there is an embedding, then we can appeal to the crystalline Torelli theorem (Theorem \ref{thm:cTorelli})
or to our explicit geometric models (Section \ref{s:models}) to realize the symplectic $G$-action
(and the action of all subgroups of $G$).

In more detail, the arguments are organized as follows:
\begin{enumerate}
\item[Section \ref{s:cond}]
develops fairly general lattice criteria
which will be applied throughout the remainder of this paper;
with this, many group actions can be excluded immediately 
 (recorded in \eqref{eq6.1}, \eqref{eq6.2}, \eqref{eq6.3}).

\item[Section \ref{s:rk3}]
treats the case $\rank\Lambda^G=3$,  recovering in particular Kond\=o's examples from \cite{Kondo}.
For use in characteristic $p=2$, we also develop an analogue of the Witt cancellation theorem (\S \ref{ss:Witt}),
 leading to Table \ref{table2}.

\item[Section \ref{s:rk4=tame}]
covers the tame groups $G$ with $\rank \Lambda^G=4$
 (Theorem \ref{thm:tame}, Table \ref{tab});
at the same time, this provides an independent confirmation of the results in \cite{DKeum}.

\item[Section \ref{s:rk4=wild}]
treats  the wild groups $G$ with $\rank \Lambda^G=4$  (with $p\leq 11$),
thus completing the analysis in rank 4.
{ The main subtlety is that also subgroups of $M_{23}$ with 4 orbits and $\Lambda^G$ of square determinant  show up
(which were ruled out in the tame case by condition \eqref{eq:Legendre}, see \ref{ss:square}),
but we also check for each other subgroup of $M_{23}$ with 4 orbits (from Table \ref{tab})
in which small characteristics it is realized on $X_{0,p}$ (Table \ref{table7}).
}

\item[Section \ref{s:rk5}]
covers all remaining cases with $\rank \Lambda^G\geq 5$
needed for our main theorems: 
with all the preliminary work, it is easy to verify that  they all can be realized on $X_{0,p}$ if $p>2$
(Proposition \ref{prop:Mukai}).
We also give another independent proof of Mukai's results \cite{Mukai} over $\CC$
 (Theorem \ref{thm:equiv} in \S \ref{ss:Mukai}).

\item[Section \ref{s:models}]
completes the ingredients needed for the proofs of our main theorems
by exhibiting explicit realizations of the required group actions on the supersingular K3 surfaces 
of Artin invariant $\sigma_0=1$
in characteristics 2 (\S\S \ref{ss:2}, \ref{s:Aut(S_6)}) and 3 (\S \ref{ss:3})
where no crystalline Torelli theorem is available.

\item[Section \ref{s:proofs}]
collects the proofs of all theorems from the introduction and concludes the paper.
\end{enumerate}

\section{Lattice conditions}
\label{s:cond}

In this section, we collect useful criteria for immediate and later application.
 The first criteria work for  supersingular K3 surfaces $X$ of arbitrary Artin invariant $\sigma_0$
by sorting out those $G$-actions on $\Lambda$ such that there may be a primitive embedding
$\Lambda_G\hookrightarrow\NS(X)$, as required for a geometric $G$-action on $X_{0,p}$ by 
Theorem \ref{cor:rk<=3}.
For future use, we call a lattice $L$ $p$-unimodular if $p\nmid\det(L)$,
i.e.\ if the $p$-adic lattice $L\otimes_{\ZZ} \ZZ_p$ is unimodular.

\subsection{Length criterion}
\label{ss:length}

\begin{lemma}
\label{lem:p_vs_rk}
Let $G$ be a subgroup of $O(\Lambda)$.
If there is some prime $p$ such that 
\begin{eqnarray}
\label{eq:rk-2}
l_p(\Lambda_G)>\rank \Lambda^G-2,
\end{eqnarray}
then the $G$-action on $\Lambda$ may only be realized by a symplectic action on a supersingular K3 surface
%$X_{0,p}$
%with Artin invariant $\sigma_0=1$ 
in characteristic $p$.
\end{lemma}

\begin{proof}
This follows directly from the analogue of Corollary \ref{length-condition} for arbitrary Artin invariant $\sigma_0$
as $\NS(X)$ is always $p'$-unimodular at any prime $p'\neq p$.
Indeed, if $G$ acts on $X$ in characteristic $\neq p$,
then $\Lambda_G$ glues to its orthogonal complement inside $\NS(X)$ along their $p$-Sylow groups,
so
\[
l_p(\Lambda_G) = l_p((\Lambda_G)^\perp_{\NS(X)}) \leq \rank((\Lambda_G)^\perp_{\NS(X)}) = \rank(\Lambda^G)-2,
\]
contradicting \eqref{eq:rk-2}.
\end{proof}

\begin{corollary}
\label{cor:pvsq}
If some $G$-action on $\Lambda$ admits two primes as in Lemma \ref{lem:p_vs_rk},
then it does not correspond to
a symplectic action on any supersingular K3 surface.
\end{corollary}

Lemma \ref{lem:p_vs_rk} determines $p$ in several cases of small rank (see Section \ref{s:rk3}).
Meanwhile Corollary \ref{cor:pvsq} directly rules out the $G$-actions from \cite{HM290} with no.s
\begin{eqnarray}
\label{eq6.1}
\begin{matrix}
62, 96, 140, 142, 150, 151, 155 - 159, 161, 176, 187, 191,
\\ 195,
 198, 200, 204, 206, 208 - 211, 214, 217 - 220.
 \end{matrix}
\end{eqnarray}

\subsection{Constituent criterion}
\label{ss:constituent}

\begin{lemma}
\label{lem:constituent}
Assume that there is some prime $p$ such that $l_p(\Lambda_G) = \rank \Lambda^G$.
If 
\begin{eqnarray}
\label{eq:ee}
\sum_{e>1} n_e > \rank \Lambda^G-2
\end{eqnarray}
where the $n_e$ are as in Lemma \ref{-1} for $\Lambda^G$ or equivalently $\Lambda_G$,
then the $G$-action on $\Lambda$ cannot 
correspond to a symplectic action on 
a supersingular K3 surface. %$X_{0,p}$.
\end{lemma}

%\begin{lemma}
%If $l_2(A(\Lambda_G))=3$, then the $2$-adic Jordan constituent for $p=2$ has dimension at least $2$.
%\end{lemma}

\begin{proof}
Let $r=\rank \Lambda^G$, exceeding the rank of the purported invariant lattice $\NS(X_0)^G$  by 2,
so that any supersingular K3 surface $X$ with a symplectic $G$-action lives in characteristic $p$ by Lemma \ref{lem:p_vs_rk}.
%(and the reduction step in Proposition \ref{prop:reduction}).
Denote the minimal orders of generators of the Sylow group $A(\Lambda_G)_p$ by $q_1\leq \hdots \leq q_r$.
Here $q_2>p$, since otherwise $n_1\geq 2$ would give, using \eqref{eq:ee},
\[
l_p(\Lambda_G) = \sum_{e\geq 1} n_e > \rank(\Lambda^G) \geq l_p(\Lambda^G),
\]
contradicting the equality of the two $p$-lengths.
To embed into $\NS(X_{0})$,
the coinvariant lattice $\Lambda_G$ has to glue to $\NS(X)$, of rank $r-2$,
along a subgroup of $A(\Lambda_G)$ 
whose length does not exceed $r-2$.
For length reasons, it follows that the resulting overlattice $\NS(X)$ has some constituents with $q_1'\geq q_1, q_2'\geq q_2$
but then it cannot be $p$-elementary since $q_2>p$ by what we have seen above.
\end{proof}

Applied to the group actions from \cite{HM290} at the prime $p=2$, we can rule out the following no.s:
\begin{eqnarray}
\label{eq6.2}
\begin{matrix}
57, 64,  78, 86, 93, 107, 117, 130, 132, 143, 147, 153, 160, 164, 174,\\
177, 181, 184, 185, 190, 192, 196, 197, 202, 207, 213, 215, 216, 221.
\end{matrix}
\end{eqnarray}

%\begin{corollary}
%\label{cor:constituents}
%The groups with no.s 164, 174, 177, 181, 184, 185, 190, 192, 196, 197, 202, 207, 213, 215, 216, 221 
%cannot act symplectically on any supersingular K3 surface.
%\end{corollary}
%

\begin{remark}
At $p=2$, the genus symbols in \cite{HM290}
do not always attain the correct form according to \cite[\S 15.7]{CS99} (cf.\ the discussion in \cite{Allcock-memoir}),
but we correct this on the way whenever necessary.
For instance, for no.\ 197, \cite{HM290} lists
$2_3^{+1}4_1^{+1}16_7^{+1}$,
but the trace of the order $2$ constituent is clearly not compatible with the determinant
as $3$ is a non-square modulo $8$.
The correct genus symbol reads $2_1^{+1} 4_1^{+1} 16_1^{+1}$.
Anyway, for this case, the precise values turn out to be irrelevant 
for applying Lemma \ref{lem:constituent}.
\end{remark}

\subsection{Determinant condition}
\label{ss:det}

\begin{lemma}
\label{lem:det}
Assume that $\Lambda_G\hookrightarrow\NS(X_{0,p})$
such that $l_{p'}(\Lambda_G^\perp) = \rank \Lambda^G-2$
at some prime $p'$. 
Write $d'$ for the prime-to $p'$-part of $-\det \Lambda_G$.
Then
\begin{eqnarray}
\label{eq:det-cond-p'}
\left(\frac{d'}{p'}\right) = \prod_{q'>1} \varepsilon_{q'}
\end{eqnarray}
where the product runs over the signs of the $p'$-adic constituents of $\Lambda_G^\perp$.
\end{lemma}

\begin{proof}
Since $\det(\NS(X_{0,p}))=-p^2$, the square classes of the determinants of $\Lambda_G$ and $\Lambda_G^\perp$
agree up to sign. Thus $d'$ gives the prime-to $p'$-part of the square class of $\Lambda_G^\perp$,
and \eqref{eq:det-cond-p'} is just the standard $p'$-adic determinant condition  (cf.\ \cite[p.\ 383]{CS99}).
Here we use that 
$\Lambda_G^\perp$ has  no unimodular $p'$-adic constituent ($q'=1$)
because rank and length agree by assumption.
\end{proof}

The determinant condition can also be applied in case $p=p'$,
but it is most readily used when $p'\neq p$ (where $p$ could be determined by Lemma \ref{lem:p_vs_rk})
 and $l_{p'}(\Lambda_G)=\rank(\Lambda^G)-2$,
for then  the $p'$-adic genus symbol of $\Lambda_G^\perp$ is obtained from that of $\Lambda_G$
as in Lemma \ref{-1} (with the usual extra care because of sign walking if $p'=2$).
This directly excludes the $G$-actions from \cite{HM290} with no.s
%\marginpar{I put 138 and 139 here}
\begin{eqnarray}
\label{eq6.3}
104, 114, 127, 136,  138, 139, 144, 166, 168, 173, 186, 189,  199.
\end{eqnarray}
(For no.\ 186 with symbol $4_1^{+1}3^{+3}$, 
we note that the single non-unimodular $2$-adic constituent $4_1^{+1}$ does not allow for any sign walks
by \cite[pp.\ 381/2]{CS99},
so $\Lambda_G^\perp$ would  have symbol $4_7^{+1}$, but $(-3/2)=-1$.)

We will see more applications of this criterion throughout this paper,
notably in Proposition \ref{prop:tame}.
We emphasize that it can equally well be applied to the primitive closure
of $\Lambda_G\oplus H^{(p)}$ embedding into $L_{1,25}$.
This will be especially useful in characteristic $p=2$, see \ref{ss:p=2,rk=3}.

\section{Classification in Rank 3} 
\label{s:rk3}

Let $G$ be a group acting on $\Lambda$ such that rank$(\Lambda^G)=3$.
For the  cases remaining after \ref{ss:length}, \ref{ss:constituent}, \ref{ss:det}, we will proceed with a close inspection of the discriminant groups.

%
%Thanks to the crystalline Torelli theorem (or the  explicit realizations in Section \ref{ss:models}),
%it will suffice to go through the list of \cite{HM290}
%and check whether $G$ satisfies the lattice criteria from \ref{thm:cTorelli}
%to admit a symplectic action on the supersingular K3 surface $X_{0}$
%of Artin invariant $\sigma_0=1$ in characteristic $p$ by the reduction step.
%Necessarily, this requires $\Lambda_G$ to glue to a single positive vector $v$ 
%(the purported generator of $\NS(X_0)^G$) to give the lattice $\NS(X_0)$.
%By Corollary \ref{length-condition} (ii), this rules out all groups $G$
%such that $\Lambda_G$ has $p$-length at least two at two different primes.
%This applies to no.s 176, 187, 191, 195, 198, 200, 204, 206, 208 -- 211, 214, 217 -- 220
%from \cite{HM290}.
%Overall we distinguish three cases:
%

\subsection{The discriminant group  $A(\Lambda_G)$ has length $1$}
\label{ss:length=1}

In this case, the characteristic of $k$ a priori can be any prime $p$ dividing $\det(\Lambda_G)$,
but for $\Lambda_G$ to glue to some positive vector $v$ (with $v^2=|\det(\Lambda_G)|$) to give $\NS(X_0)$,
the quadratic forms on the $p'$-Sylow groups of the discriminant groups of $\Lambda_G$ and $\ZZ v$ have to agree up to sign
for any prime $p'\neq p$.
It turns out that for each $G$-action with discriminant group  $A(\Lambda_G)$ of length $1$, 
this determines the characteristic $p$ uniquely.
This is easily checked using the genus symbol as indicated in Table \ref{table1}.

\begin{table}[ht!]
$$
\begin{array}{c|c|c|c|c|c|c}
\text{no.} & G & \text{genus symbol of } \Lambda_G & v^2 & \text{genus symbol of } \ZZ v & \text{glues at} & p\\
\hline
%163 & U_4(3) & 4_7^{+1}3^{+2} & 4 & 4_1^{+1} & 2 & 3\\
%&&&&&&&\\
165 & M_{22} & 4_5^{-1}11^{+1} & 44 & 4_3^{-1}11^{+1} & 2 & 11\\
170 & M_{21}.2_2 & 4_3^{-1}3^{-1}7^{-1} & 84 & 4_5^{-1}3^{+1}7^{-1} & 2,3 & 7\\
172 & 2^4.\mathfrak A_7 & 8_1^{+1}7^{+1} & 56 & 8_7^{+1}7^{+1} & 2 & 7\\
175 & \mathfrak A_8 & 4_1^{+1}3^{+1}5^{+1} & 60 & 4_7^{+1}3^{+1}5^{-1} & 2,3& 5\\
182 & M_{11} & 
\begin{cases}
2_7^{+1}3^{-1}11^{-1}\\
2_3^{-1}3^{-1}11^{-1}\end{cases} & 66 & 2_1^{+1}3^{+1}11^{-1} & 
\begin{cases}2,3\\
3
\end{cases} & 
\begin{cases}
11\\ -
\end{cases}
\\
183 & 
2^4 :(3\times\mathfrak A_5):2 & 8_1^{+1}3^{-1}5^{-1} & 120 & 8_7^{+1}3^{+1}5^{+1} & 2,3 & 5\\
\end{array}
$$
\caption{6 groups acting symplectically on some $X_{0,p}$}
\label{table1}
\end{table}

\begin{summary}
The  groups in  Table \ref{table1} admit no symplectic action on any supersingular K3 surface
in characteristic $p'\neq p$ (the prime indicated in the last column),
but by Theorem \ref{thm:cTorelli} they can be realized symplectically in characteristic $p$.
%Alternatively, the last claim follows without reference to the crystalline Torelli Theorem  
%by exhibiting explicit models (see Section \ref{ss:models}).
\end{summary}

\begin{remark}
Kond\=o asserted this already for the simple and perfect groups in \cite{Kondo}.
However, he follows \cite[Table 10.3]{CS99} to list one of the maximal subgroups as
$2^4:(3\times \mathfrak S_5)$.
Since this group contains elements of order $12$, but $M_{23}$ doesn't (cf.\ Table \ref{tab:23}),
this cannot be correct. %, as in Remark \ref{rem:DK}.
The group structure given in \cite{Atlas}, 
$2^4 :(3\times\mathfrak A_5):2$,
appears to be correct
(while it is listed as $[2^43].S_5$
under $\#183$ in \cite{HM290}).
\end{remark}

\subsection{There is an odd prime $p$ such that $l_p(\Lambda_G)\geq 2$}

By Corollary \ref{length-condition} (ii), this implies that char$(k)=p$,
and the quadratic forms of the $p'$-Sylow groups of the discriminant groups of $\Lambda_G$ and $\ZZ v$ agree up to sign
for any prime $p'\neq p$ (as in \ref{ss:length=1}). (In particular, $l_{p'}(\Lambda_G)\leq 1$.)

\subsection*{\HM 163}
$G=U_4(3), \Lambda_G\in 4_7^{+1}3^{+2}$ leads directly to $p=3$ and $v^2=4$,
satisfying the requirement.
The symplectic action is realized geometrically on the Fermat quartic, see Section \ref{s:models}.

\subsection*{\HM 167}
$G=U_3(5)$, $\Lambda_G\in \II_{0,21}2^{+1}_7 5^{-2}$
directly leads to $p=5$ and $v^2=2$.
This is immediately verified to meet the requirement
(and thus admits a symplectic action on the supersingular K3 surface of Artin invariant $1$ 
in characteristic $p=5$)
by Theorem \ref{thm:cTorelli}
(realized in \cite[\S 6]{DKeum}, for instance).

\subsection*{\HM 169}
$G=3^4.A_6.2$, $\Lambda_G\in \II_{0,21}2^{+1}_1 3^{-1}9^{+1}$
 implies that char$(k)=3$ and $v^2=6$,
but then all order 3 elements in $w\in A(\Lambda_G)$ have square $w^2\equiv\frac 23\mod 2\ZZ$ -- just like $v/3$ -- or  $w^2\equiv 0\mod 2\ZZ$,
so in either case, they cannot glue to $v/3$, and this group is excluded.

%\subsection*{\HM 186}
%$G=[2^43^5], \Lambda_G\in\II_{0,21} 4_1^{+1}3^{+3}$, so $p=3$ and $v^2=12$,
%and the lattices do not glue along the $2$-Sylow  group of the discriminant group
%since the genus symbol of $\ZZ v$ is $\II_{1,0}4_3^{-1}3^{+1}$.
%% $q=2$.

\subsection*{\HM 201}
$G=[2^43^3], \Lambda_G\in\II_{0,21} 2_7^{+1}3^{+2}9^{-1}$,
so $p=3$ and $v^2=18$. Here it is more straight forward  %to argue in $U\oplus \Lambda$
to first glue $\Lambda_G$ to $H^{(p)}$ inside $U\oplus \Lambda$. We get
\[
\II_{0,21} 2_7^{+1}3^{+2}9^{-1} \oplus \II_{0,4} 3^{+2} \rightsquigarrow
\II_{0,25}  2_7^{+1}9^{-1}
\]
which visibly does not glue to $v$ at $9$ (because $\ZZ v\in\II_{1,0} 2^{+1}_19^{-1}$).

\begin{summary}
This completes the classification of symplectic group actions with rank$(\Lambda_G)=3$ outside characteristic $2$.
\end{summary}

\subsection{Groups with $l_2(\Lambda_G)\geq 2$}
\label{ss:p=2,rk=3}

Having ruled out plenty of the groups from \cite{HM290}
such that rank$(\Lambda^G)=3$, we are left with issues in characteristic $p=2$
where $l_2(A(\Lambda_G))\geq 2$. 
Assuming that the $G$-action comes from the supersingular K3 surface $X_{0,2}$,
we consider the auxiliary lattice $L'$ which is the saturation
of $\Lambda_G\oplus H^{(2)}$ inside $L_{1,25}$:
\[
L' = (\Lambda_G\oplus H^{(2)})' \cap L_{1,25}.
\]
Since $L'$ has orthogonal complement generated by a single positive vector $v$ (as before in this section),
we infer that $\Lambda_G$ and $H^{(2)}$ glue along the full discriminant group $A(H^{(2)}) \cong v(2)$
in standard notation. 
This has immediate consequences, the first of which is valid in greater generality:

\begin{lemma}
\label{lem:odd-const}
Assume that the $2$-adic lattice $L$ has 
either  no constituent at $2$ or 
an odd constituent at $2$ of dimension $1$ or $2$,
or $2$-adic symbol $2_\II^{+2} \hdots$ without odd constituents at level $4$.
Then $v(2)\not\hookrightarrow A(L)$.
\end{lemma}

\begin{proof}
For $v(2)$ to embed into $A(L)$,
we need 2 odd integral 2-torsion elements in $A(L)$ with half-integer intersection.
In the first alternatives, however, all odd integral 2-torsion elements are orthogonal in $A(L)$.
The third alternative admits essentially only one odd integral 2-torsion element;
more precisely, we distinguish the  summand $u(2)$ of $A(L)$ corresponding to the level $2$-constituent,
and its orthogonal complement $u(2)^\perp$.
By construction, any  two-torsion elements elements $u_1, u_2\in u(2)^\perp$ are even integral and orthogonal
while $u(2)$ contains a unique odd integral class $u$.
Hence any two odd integral two-torsion elements can be written as $u+u_1, u+u_2$ and continue to be orthogonal.
In each case, we thus derive a contradiction.
\end{proof}

\begin{summary}
The lemma allows us to rule out several cases in rank $3$, and also analogously in higher rank
whenever $l_2(\Lambda_G)>\rank(\Lambda^G)-2$, namely no.s
%\marginpar{I omitted 139 here, now at det crit.}
\[
48, 56, 58, 80, 91, 95, 
116, 123, 131, 148, 152, 154, 
178, 180, 193, 203, 205, 212.
\]
\end{summary}

\subsection{Witt cancellation theorem}
\label{ss:Witt}

On the other hand, if $v(2)$ embeds into $A(\Lambda_G)$,
then it will be very useful to know its complement.
In fact, the following analogy of the Witt cancellation theorem for finite quadratic forms holds in our case.
Recall that there are two building blocks of finite quadratic forms, 
$$u(2)=A(2^{+2}_{\II}), \;\;\;
v(2)=A(2^{-2}_{\II}),
$$
whose quadratic forms take integer values. 
The underlying groups are elementary abelian of order 4 for both cases and the Gram matrix is given by 
\[
u(2)=\begin{pmatrix}0&1/2\\1/2&0\end{pmatrix},\quad v(2)=\begin{pmatrix}1&1/2\\1/2&1\end{pmatrix},
\]
where the diagonal (resp. off-diagonal) entries are considered as elements in $\mathbb{Q}/2\mathbb{Z}$
(resp. $\mathbb{Q}/\mathbb{Z}$).
(For instance, they can be realized as discriminant forms of $U(2)$ or $D_8$ resp.\ $D_4$.)

\begin{lemma}
\label{lem:Witt}
Let $A$ be a nondegenerate finite quadratic form and $Q$ be a fixed direct sum of $u(2)$ and $v(2)$ 
such that there exists an embedding
$Q\hookrightarrow A$. Then, $Q^\perp\subset A$ is determined uniquely up to isometries by $A$ and $Q$,
and is independent of the chosen embedding.
\end{lemma}

\begin{proof}
First we recall the theorems which we utilize. One is  
Nikulin's strong approximation theorem \cite[Theorem 1.9.1]{Nikulin80}, which states that for a given 
finite quadratic form $A$ on a $p$-group, there exists a unique $p$-adic lattice $L$ such that 
$A(L)\simeq A$ and $\mathrm{rank}(L)=\mathrm{length}(A)$, 
except when discriminant form restricted to the 2-torsion subgroup of $A$ is not integer-valued. 
In the latter case, the theorem moreover states that there are exactly two $p$-adic lattices with these properties.
The second input is the definition and existence of (refined) canonical symbols for $2$-adic lattices by \cite{AGM}. 
It corrects the one given in \cite[\S 15.7.4]{CS99} as follows. Instead of trains in \cite[\S 15.7.5]{CS99}, the authors of \cite{AGM} introduce {\em{signways}} to precisely describe the interactions of signs at various scales under sign-walkings. Signways subdivide trains at certain compartments of dimension 2, see \cite[Section 6]{AGM}. Thus, the (refined) canonical symbol is defined by the property that it has at most one minus sign in each signway, and it is at the leftmost scale of the way if exists. 
Then two 2-adic lattices are isometric if and only if they have the same (refined) canonical symbols.

Now, let $A,Q$ be given as in the statement. By the Jordan decomposition of general finite forms into the orthogonal sum of $p$-Sylow subgroups, it is enough to show the lemma at $p=2$. So, we may assume that $A$ is a 2-group.
Since $Q$ is nondegenerate, $A=Q\oplus Q^{\perp}$ holds as finite quadratic forms. By strong approximation, 
there are 2-adic lattices $L,M,N$ such that
$A=A(L),\ Q=A(M),\ Q^{\perp}=A(N)$
and whose respective ranks are equal to the lengths of $A,Q,Q^{\perp}$. By assumption, we have $M=2^{em}_{\II}$ for some $e\in \{\pm 1\},\ m\in 2\mathbb{Z}_{>0}$ uniquely. We have to separate the cases in some detail.

\subsubsection{}
If the $2$-torsion subgroup of $A$ is integer-valued, and hence so is $Q^{\perp}$, then the 2-adic 
lattices $L$ and $N$ are determined uniquely, too. We consider the canonical 2-adic symbol of $N$. 

\subsubsection{}
If it begins with one of the following expressions,
\begin{equation}\label{witt-0} q_X^{\varepsilon n}\cdots,\quad 4_{\II}^{\varepsilon n}\cdots,\quad 2_{\II}^{\varepsilon n}\cdots,\end{equation}
where $q\geq 8,\ X\in \mathbb{Z}/8\mathbb{Z}\cup \{\II\}, \varepsilon\in\{\pm 1\}$ and $n>0$, 
then the canonical symbol of $M\oplus N$ will be
\begin{equation}\label{witt-1} 
2^{em}_{\II}q_X^{\varepsilon n}\cdots,\quad 2_{\II}^{em}4_{\II}^{\varepsilon n}\cdots,\quad 2_{\II}^{(e\varepsilon) (m+n)}\cdots
\end{equation}
respectively. These are the cases where the signway of $N$ {\em{does not extend}} in $M\oplus N$, in the sense that 
for each signway of $N$, the set of its scales constitute an (entire) signway in $M\oplus N$ again.
In fact, in the former two cases of $(\ref{witt-1})$ the constituent $2^{em}_{\II}$ is itself a signway and the remaining symbols and its signways are the same as that of $N$. In the last case, the signways of $N$ and 
$M\oplus N$ correspond with the same set of scales.
These observations ensure that the omitted symbols in $\cdots$ are completely the same in $(\ref{witt-0})$ and $(\ref{witt-1})$.

\subsubsection{}
On the other hand, if the canonical symbol of $N$ begins with an odd constituent of scale 4, i.e.\ $N=[4^{\varepsilon n}\cdots$ (where $[$ denotes the beginning of a compartment), then the canonical symbol of $M\oplus N$ will be either 
\begin{equation}\label{witt-2}
2^{em}_{\II}[4^{+ n}\cdots \text{\quad or\quad }2^{(-e)m}_{\II}[4^{+ n}\cdots
\end{equation}
according as $\varepsilon=+1$ or $-1$. This is the case where the signway of $N$ {\em{extends}} in 
$M\oplus N$, as the new constituent of scale 2 becomes a part of a signway which extends the signway of $N$ that contains the least scale 4. We note that even in this case, the omitted symbols in $\cdots$ are 
unchanged from those of $N$.
Comparing the possible canonical symbols (\ref{witt-1}) and (\ref{witt-2}) case-by-case, we see that the isomorphism class of $N$ is uniquely determined by that of $M\oplus N$. In turn, the 2-adic lattice $M\oplus N\simeq L$ is uniquely determined from $A$ as we noticed before. Hence $Q^{\perp}=A(N)$ is also uniquely determined.

\subsubsection{}
Next we consider the case where the $2$-torsion subgroup of $A$ is not integer-valued, and so isn't $Q^{\perp}$. 
In this case, the strong approximation theorem states that there correspond two 2-adic lattices to each 
$A$ and $Q^{\perp}$. We choose one representative $N$ for $Q^{\perp}$.
Since $Q$ is even, 
we know that the canonical symbol of $N$ begins with odd constituent of scale 2 at which its first compartment begins, so it is
\begin{equation}\label{witt-2a} 
[2^{\varepsilon n}\ddda ]_t\dddb\ \ (t\in \mathbb{Z}/8\mathbb{Z},\varepsilon\in\{\pm 1\}, n>0),
\end{equation}
where $\ddda$ and $\dddb$ are appropriate symbols, possibly empty.
Again, we separate the cases by distinguishing whether signways of $N$ extend in $M\oplus N$ or not. 
A careful inspection of cases (see \cite[Lemma 6.1]{AGM}) shows that the signway of $N$ extends in $M\oplus N$ if and only if 
\begin{equation}\label{witt-2b}
\begin{split}
&\text{$N$ has 1-dimensional constituents at scales 2 and 4 and they }\vspace{-2mm}\\
&\text{ constitute a 2-dimensional compartment of total oddity $0$ or $4$}.
\end{split}
\end{equation}
If this is not the case, the canonical symbol of $M\oplus N$ is given simply by
\begin{equation}\label{witt-3}
[2^{(e\varepsilon)(m+n)}\ddda ]_t\dddb.
\end{equation}
If $(\ref{witt-2b})$ holds true, the canonical symbol of $N$ is one of 
%\[ [2^{\varepsilon 1}4^{+1}]_s\cdots,\quad , s=2,6
\[ [2^{\varepsilon 1}4^{\varepsilon 1}]_0\dddb,\quad [2^{\varepsilon 1}4^{(-\varepsilon) 1}]_4\dddb,\]
with $\varepsilon=\pm 1$. The canonical symbol of $M\oplus N$ then becomes 
\begin{equation}\label{witt-4}
%[2^{(e\varepsilon) (1+m)}4^{+1}]_s\cdots ,\quad 
[2^{e(1+m)}4^{+1}]_{2-2\varepsilon}\dddb ,\quad [2^{(-e)(1+m)}4^{+1}]_{2-2\varepsilon}\dddb
\end{equation}
respectively. Now, our lemma follows if the two possible representatives $L,L'$ of $A$ both lead to (possibly several $N$ but) the same discriminant group $A(N)$. 

%We check therefore the coincidence of symbols in (\ref{witt-3}) and (\ref{witt-4}). 
%If the condition (*) is satisfied, and 
\subsubsection{}
Suppose that 
$L$ has the symbol (\ref{witt-3}) (and that \eqref{witt-2b} does not hold true). Then, the other representative $L'$ is given by
\[[2^{(-e\varepsilon)(m+n)}\ddda ]_{t+4}\dddb\]
by the strong approximation theorem. (Note that this is just (\ref{witt-3}) applied with ``fake'' sign-walking at scales 1 and 2. The same holds in general, see \cite{Nikulin80}). 
By inspection of the possible symbols (\ref{witt-3}) and (\ref{witt-4}), 
we can recover the  lattices $N, N'$ as follows,
omitting illegal symbols right away:

\begin{table}[ht!]
$$
\begin{array}{c|c}
(\varepsilon, n ,\ddda, t) & \text{possible } N,N' \\
\hline
(1,1,4^{+1},0) & [2^{+1}4^{+1}]_0\dddb,\quad [2^{-1}4^{+1}]_4\dddb \\
(1,1,4^{+1},4) & [2^{-1}4^{-1}]_0\dddb,\quad [2^{+1}4^{-1}]_4\dddb\\%[2^{+1}4^{+1}]_4\dddb,   is illegal! \\
(-1,1,4^{+1},0) & [2^{+1}4^{-1}]_4\dddb,\quad [2^{-1}4^{-1}]_0\dddb\\%[2^{-1}4^{+1}]_0\dddb,   is illegal! \\
(-1,1,4^{+1},4) & [2^{-1}4^{+1}]_4\dddb,\quad [2^{+1}4^{+1}]_0\dddb\\
\text{(otherwise)} & [2^{\varepsilon n}\ddda]_{t}\dddb,\quad [2^{-\varepsilon n}\ddda]_{t+4}\dddb\\
\end{array}
$$
\end{table}

In each case, the strong approximation theorem implies again that the two possible entries for $N, N'$
belong to the same discriminant form $A(N) \cong A(N')$.
Thus the lemma follows in this case.

\subsubsection{}
It remains to consider the case where $L$ has the symbol $(\ref{witt-4})$ (and where \eqref{witt-2b} holds true as well). 
Then the corresponding $L'$ has the symbol
\begin{equation}\label{witt-5}
[2^{(-e)(1+m)}4^{+1}]_{6-2\varepsilon}\dddb ,\quad [2^{e(1+m)}4^{+1}]_{6-2\varepsilon}\dddb
\end{equation}
respectively. Again we go through the possibilities to check the next table.

\begin{table}[ht!]
$$
\begin{array}{c|c|c}
\varepsilon & L,L' \text{ (omitted $\dddb$)} & \text{possible } N,N' \text{ (omitted $\dddb$)}\\
\hline
1 & [2^{e(1+m)}4^{+1}]_0,\quad [2^{(-e)(1+m)}4^{+1}]_4 & [2^{+1}4^{+1}]_0,\quad [2^{-1}4^{+1}]_4\\
-1& [2^{e(1+m)}4^{+1}]_4,\quad [2^{(-e)(1+m)}4^{+1}]_0 & [2^{-1}4^{-1}]_0,\quad [2^{+1}4^{-1}]_4\\
1 & [2^{(-e)(1+m)}4^{+1}]_0,\quad [2^{e(1+m)}4^{+1}]_4 & [2^{+1}4^{-1}]_4,\quad [2^{-1}4^{-1}]_0\\
-1& [2^{(-e)(1+m)}4^{+1}]_4,\quad [2^{e(1+m)}4^{+1}]_0 & [2^{-1}4^{+1}]_4,\quad [2^{+1}4^{+1}]_0\\
\end{array}
$$
\end{table}

In every case, we see as before that $A(N)\cong A(N')$ is uniquely determined as a finite quadratic form.
This completes the proof of Lemma \ref{lem:Witt}.
\end{proof}

\subsection{Remark}
The Witt cancellation theorem has been known to hold in several other situations since the classical theory of 
quadratic forms over fields. We want here to exhibit an example in which the cancellation {\em{fails}}
for general embeddings of finite quadratic forms. 
Thus, the assumption on the quadratic form of $Q$ attaining integer values is really necessary.

\begin{example}
Let $A$ be the finite quadratic form defined by the following underlying group and the Gram matrix,
\[A=\left((\ZZ/2\ZZ)^2\oplus (\ZZ/4\ZZ)^2,\ 
\begin{pmatrix}1/2&\\&3/2\end{pmatrix}\oplus
\begin{pmatrix}3/4&\\&3/4\end{pmatrix}
\right).\]
We consider the embedding of the following form 
\[Q=\left( (\ZZ/2\ZZ)^2,\ \begin{pmatrix}1/2&\\&3/2\end{pmatrix}\right)\]
into $A$. We denote by $e_1,e_2,f_1,f_2$ the basis of $A$ corresponding to the above matrices, i.e.,
$e_i$ has order 2 $(i=1,2)$ and $f_i$ has order 4 $(i=1,2)$,
$(e_1)^2=1/2,\ (e_2)^2=3/2,\ (f_i)^2=3/4\ (i=1,2)$ and every two distinct elements are orthogonal.
We then find two subspaces isomorphic to $Q$ by
\begin{equation*}
Q_1=\lrbracket{e_1,\ e_2},\quad Q_2=\lrbracket{e_1+2f_1,\ e_2+2f_2}.
\end{equation*}
We compute their orthogonal complements as follows.
\[\begin{split}
Q_1^{\perp}=\lrbracket{f_1,f_2}\simeq \left( (\ZZ/4\ZZ)^2,\ \begin{pmatrix}3/4&\\&3/4\end{pmatrix}\right),\\
Q_2^{\perp}=\lrbracket{e_1+f_1,e_2+f_2}\simeq \left( (\ZZ/4\ZZ)^2,\ \begin{pmatrix}5/4&\\&1/4\end{pmatrix}\right).
\end{split}\]
It is easy to see that the quadratic form of $Q_1^{\perp}$ does not attain the value $1/4\in \QQ/2\ZZ$. 
Therefore, the embeddings $Q_1$ and $Q_2$ do
not have isometric complements, contrary to the situation of Lemma \ref{lem:Witt}.
%not meet the cancellation theorem.
\end{example}
%\end{remark}

\subsection{Conclusion of the classification in rank $3$}
With the help of Lemma \ref{lem:Witt} ,
we can easily continue by computing $v(2)^\perp$ for the five remaining cases in rank 3
and check whether $L'$ glues to $v$ in Table \ref{table2}.

\begin{table}[ht!]
$$ 
\begin{array}{c|c|c|c|c|c}
\text{no.} & G & \text{genus symbol of } \Lambda_G & \text{genus symbol of } L' & v & \text{glues?} \\
\hline
162 & 2^9.M_{21} & 2_\II^{-2}8_3^{-1} & 8_3^{-1} & 8 & no\\
171 & M_{21}.2_1 & 2_5^{+3}7^{-1}
& 2_5^{-1} 7^{-1} & 14  & ok \\
179 & 2^4:\mathfrak S_6 & 2_\II^{-2} 8_1^{+1} 3^{+1} & 8_1^{+1} 3^{+1} & 24 & no\\
188 & 2^5.\mathfrak S_5 & 2_\II^{-2} 8_7^{+1} 5^{-1} & 8_7^{+1} 5^{-1} & 40 & no\\
194 & \Aut(\mathfrak S_6) & 2_5^{+3} 3^{-1} 5^{+1} & 2_5^{-1} 3^{-1} 5^{+1} & 30 & ok
%166 & [2^9].A_6 & 2_\II^{-2}4_1^{+1}3^{+1} &
%12 & 4_3^{-1}3^{+1} & 3\\
%168 & [2^9].S_5 & 2_\II^{-2}4_7^{+1}5^{-1} &
%20 & 4_5^{-1}5^{+1} & 5\\
%173 & [2^{12}3^2] & 2_\II^{+2} 8_5^{-1}3^{-1} &
%24 & 8_3^{-1}3^{-1} & 3\\
%189 & G=[2^4].L_2(7) & 2_\II^{+2}4_1^{+1}7^{+1}  &
%28 & 4_7^{+1} 7^{+1} & 7\\
%199 & A_6.2 & 2_\II^{+2}4_1^{+1}3^{+1}5^{+1} &
%60 & 4_7^{+1}3^{-1}5^{-1} & 5\\
\end{array}
$$
\caption{Groups $G$ with $v(2)\hookrightarrow A(\Lambda_G) \; (p=2)$}
\label{table2}
\end{table}

\smallskip

As an explanation for no.s 171 and 194,
we denote orthogonal half-integral generators of $A(\Lambda_G)_2$ by $w_1, w_2, w_3$, each of square $7/2$.
Then $v(2)\subset A(\Lambda_G)$ comprises the sums $w_i+w_j$, and the two-torsion in 
$v(2)^\perp$ is generated by $w_1+w_2+w_3$ of square $5/2$ as stated (uniquely by Lemma \ref{lem:Witt}).

We conclude by pointing out that $G=M_{21}.2_1$ was realized in \cite{DK}.
Meanwhile we shall see in Section \ref{s:Aut(S_6)} 
that $ \Aut(\mathfrak S_6)$ can also be realized as a group of symplectic automorphisms on $X_{0,2}$
(preserving a polarization of degree $30$ indeed).

\section{Classification in rank 4 -- tame case}
\label{s:rk4=tame}

We turn to the tame case in rank 4 and start with 
an easy observation which surely is well-known.

\begin{lemma}
\label{lem:aux}
Let $p$ be prime and $G$ a finite group such that $p\nmid |G|$.
If $G$ acts on some $p$-unimodular lattice $L$,
then $p\nmid\det(L^G)\det(L_G)$.
\end{lemma}

\begin{proof}
Since $L^G$ and $L_G$ embed primitively as orthogonal complements into $L$,
$p$ dividing one's determinant is equivalent to $p$ dividing both.
But then $L^G$ and $L_G$ glue along a subgroup of size $p$ inside their discriminant groups,
i.e.\ there are primitive vectors $u\in L^G, v \in L_G$ such that
\[
w = \frac 1p (u+v) \in L.
\]
But then the orbit under $G$ gives
\[
L^G \in \sum_{g\in G} g(w) = \frac{|G|}p u, \;\;\; \text{ since } \sum_{g\in G} g(v) \in L_G\cap L^G = \{0\}.
\]
As $u$ was primitive in $L^G$, we infer that $p\mid |G|$ as claimed.
\end{proof}

We use the lemma for the following abstract characterization
of tame group actions in rank 4 (valid for any Artin invariant, in fact):

\begin{proposition}
\label{prop:tame}
Let $G$ be a finite group acting on $\Lambda$ such that $\Lambda^G$ has rank 4.
Then the $G$-action 
 satisfies the lattice-theoretic requirements for a tame symplectic action on some supersingular K3 surface $X$ in characteristic $p$
with $\rank \NS(X)^G=2$  from Theorem \ref{thm:cTorelli}
if and only if  the following conditions hold
\begin{enumerate}
\item[(i)]
$p\nmid |G|$;
\item[(ii)]
$A(\Lambda^G)$ has length at most two;
\item[(iii)]
$\sigma_0(X)=1$;
\item[(iv)]
$\det \Lambda^G$ is a non-square mod $p$, i.e.\  $\left(\frac{\det \Lambda^G}{p}\right)=-1$.
\item[(v)]
at every prime $p'\mid\det\Lambda^G$ such that 
$l_{p'}(A(\Lambda^G))=2$, 
write $d'$ for the prime to $p'$-part of $\det\Lambda^G$.
Then  
the signs $\varepsilon_q$ in the $p'$-adic symbol of $\Lambda_G$ 
satisfy
\begin{eqnarray}
\label{eq:det_p'}
\prod_{q'>1} \varepsilon_{q'} = \left(\frac{-d'}{p'}\right).
\end{eqnarray}
%
%there is either one $p'$-adic Jordan constituent of dimension two at some $p'$-power $q>1$,
%or there are two 1-dimensional $p'$-adic Jordan constituents at $q, q'>1$. Then 
%the signs $\varepsilon_q$ in the $p'$-adic symbol of $\Lambda_G$ satisfy
%\[
%\begin{cases}
%\varepsilon_q =  \left(\frac{-(\det \Lambda^G)/q^2}{p'}\right) \\
%\varepsilon_q \varepsilon_{q'} = -  \left(\frac{-(\det \Lambda^G)/(qq')}{p'}\right) 
%\end{cases}
%%\prod_{q>1} \varepsilon_q = (-1)^{s+1} \left(\frac{-\det \Lambda^G}{p'}\right)
%\]
%where $s$ denotes the number of $p'$-adic Jordan constituents for $q>1$.
\end{enumerate}

\end{proposition}

\begin{proof}
Let $G$  satisfy the lattice-theoretic requirements for a tame symplectic action on some supersingular K3 surface $X$ in characteristic $p$
with $\rank \NS(X)^G=2$ and $\Lambda_G=\NS(X)_G$.
Then Condition (i) is obvious,
and  $p\nmid\det \Lambda^G$ by Lemma \ref{lem:aux}.
By assumption, $\Lambda_G\cong \NS(X)_G$ glues to $M':=\NS(X)^G$,
a hyperbolic even lattice of rank two,
so
\[
l_{p'}(A(\Lambda^G)) = l_{p'}(A(\Lambda_G)) \leq \rank M' = 2
\]
holds trivially at all primes $p'\neq p$
while the $p$-length  of $A(\Lambda^G)$ is zero as we have seen. Hence (ii) follows 
 by the Chinese Remainder Theorem,
and we also infer that 
\[
l_p(\NS(X))=l_p(M')\leq \rank(M')=2.
\]
This implies (iii).

We conclude that the $p$-Sylow subgroup $A(M')_p$
is anti-isometric to $A(H^{(p)})$. 
Since the rank is even, this amounts to the $p$-adic symbol of $M'$ 
agreeing with the non-unimodular factor of the $p$-adic symbol of $H^{(p)}$, i.e. 
\begin{eqnarray}
\label{eq:pe2}
p^{\varepsilon_p 2}, \;\;\; \varepsilon_p =  - \left(\frac{-1}p\right).
\end{eqnarray}
But then by construction, we have $\det(M') = - p^2\det\Lambda^G$,
so the standard determinant condition as in \ref{ss:det} %  \cite[p.\ 383, (29)]{CS99} 
implies
\[
\varepsilon_p =   \left(\frac{-\det \Lambda^G}{p}\right).
\]
Comparing with \eqref{eq:pe2} we infer condition (iv).

Condition (v) follows in a similar way.
In fact, this is just Lemma \ref{lem:det}
because the  sign products 
$\prod_{e\geq 1} \varepsilon_{e}$ agree for $ \Lambda^G, \Lambda_G, \Lambda_G^-$ and $M'$ by Lemma \ref{-1}.
%Namely, since $M'$ glues to $\Lambda_G$ at all primes $p'\neq p$, it
% has the same $p'$-adic symbol as $\Lambda_G^-$,
% but then (under the assumption that $l_{p'}(\Lambda^G)=2$),
%% on the $p'$-length of $A(\Lambda^G)$),
% the sign products 
%$\prod_{q'>1} \varepsilon_{q'}$ agree for $ \Lambda^G, \Lambda_G, \Lambda_G^-$ and $M'$ by Lemma \ref{-1}.
%Then the compatibility with $\det M'$ gives \eqref{eq:det_p'}.

Conversely, given $G$ with $\Lambda^G$ satisfying conditions (i), (ii), (iv) and (v),
we have to make sure that $\Lambda_G$ glues to some rank two hyperbolic even lattice $M'$ 
to give $\NS(X_{0,p})$,
and check the conditions of the crystalline Torelli theorem.
By what we have seen above, the genus symbol of $M'$ is predicted as that of $\Lambda_G^-$ augmented
by a factor $p^{\varepsilon_p 2}$. 
The existence of $M'$ as an even integral lattice
can be inferred from \cite[Thm.\ 1.10.1.]{Nikulin80},
but in line with our previous arguments, we will content ourselves with
checking the conditions from \cite[15.7.7]{CS99}.
Here the compatibilities of determinants are  encoded in conditions (iii) and (iv) by the same argument as above;
for completeness we comment on the case of primes $p'\neq p$ such that $l_{p'}(A(\Lambda^G))=1$.
Here the $p'$-adic symbol contains a factor $q^{\varepsilon_e1}$ (derived from $\Lambda_G$ by Lemma \ref{-1}),
so we are free to set the other factor to be $1^{\varepsilon_01}$ where
\[
\varepsilon_0 = \varepsilon_e \left(\frac{-d'}{p'}\right)
\]
accommodating the $p'$-prime part $-d'$ of the determinant of $M'$.

We conclude by claiming that everything else from \cite[15.7.7]{CS99} needed to infer the existence of $M'$
is inherited from $\Lambda_G^-$ (which surely exists).
Indeed, the only non-trivial condition is the congruence relation of signature, $p$-excess and oddity modulo $8$:
\begin{eqnarray}
\label{eq:CS-cond}
\text{signature } + \sum_{p''>2} p''\text{-excess } \equiv \text{ oddity } \mod 8.
\end{eqnarray}
When moving from $\Lambda_G^-$ to $M'$, the signature drops from $20$ to $0$,
but other than that only the $p$-excess changes from $0$ (since $p\nmid\det\Lambda^G$)
to $2(p-1) + 2(1-\varepsilon_p)$.
But the last expression is congruent to $4$ modulo $8$ for any odd $p$,
so the two correction terms balance out modulo $8$ and \eqref{eq:CS-cond} remains valid for $M'$.
Thus the required integral lattice $M'$ exists, and 
it 
%remains to verify that $M'$ can be taken to be even.
%This can be confirmed explicitly case by case going though the classification  \cite{HM290},
%but there is an independent structural argument as follows:
%$\Lambda_G$ and $\Lambda_G^-$ are even,
%so the unimodular constituent at $p=2$ has type $1_\II^m$
%where $m$ is even by the usual parity constraints.
%This enters with trace $0$ in the oddity, so for $M'$ we can (and should) also take $1_\II^{m'}$
%for some appropriate $m'\in 2\ZZ$ (where $m'$ is even, since $\rank \Lambda_G^- \equiv \rank M' \mod 2$ 
%and the non-unimodular parts of the 2-adic symbols of $\Lambda_G^-$ and $M'$ agree by construction).
%Thus, with this genus symbol, $M'$ 
is automatically even;
by construction
  $\Lambda_G\oplus M'$ has $\NS(X_{0,p})$ as a finite index overlattice admitting the required $G$-action.
%
%
%
%This follows automatically whenever $l_2(A(\Lambda^G))=2$.
%By inspection, there is only one maximal group $G$ in \cite{HM290} with $l_2(A(\Lambda^G))<2$
%satisfying (i) -- (iv), namely $G=\mathfrak A_7$ with $\Lambda_G\in\II_{1,21}3^{+1}5^{+1}7^{+1}$.
%Consider the even forms
%\[
%M_1 = \begin{pmatrix}
%8 & 3\\
%3 & -12
%\end{pmatrix},
%\;\;\
%M_2 = \begin{pmatrix}
%4 & 3\\
%3 & -24
%\end{pmatrix}.
%\]
%Then one directly checks for any prime $p$ satisfying (iii) 
%that $\Lambda_G$ glues to
%\[
%M_2\left[\left(\frac p7\right)p\right], \;\; \text{ if } p\equiv 1,4\mod 5,\;\;\;
%\text{ and to} \;\;
%M_1\left[\left(\frac p7\right)p\right], \;\; \text{ if } p\equiv 2,3\mod 5.
%\]
%Thus we also conclude for the $G$-action in this case.
\end{proof}

\begin{remark}
\label{rem:tame-wild}
By inspection of the proof, the arguments only require that $p\nmid\det\Lambda^G$
(as inferred in the tame case from Lemma \ref{lem:aux}),
so they also apply to certain wild cases (see \ref{ss:square}, for instance). 
%
%In particular, this also confirms that all non-maximal groups $G'$ satisfying (i) -- (iv) are covered by Proposition \ref{prop:tame}.
%\marginpar{This should go somewhere else -- the prop is true anyway, but it may be needed for the next thm}
%Namely, if $G'\subset G$ with $\Lambda_G = \Lambda_{G'}$,
%then 
%\begin{itemize}
%\item
%either $p\mid\det\Lambda^G$, so both $G$ and $G'$ are wild by Lemma \ref{lem:aux}
%\item
%or 
%$p\nmid\det\Lambda^G$, so the same existence criterion applies to $G$ and $G'$
%independent of being tame or not.
%\end{itemize}
\end{remark}

We shall now apply  Proposition \ref{prop:tame} to the classification in \cite{HM290}.
In principle we can just go through the single cases,
but our main goal is to infer an  abstract group theoretic characterization in the spirit of \cite{Mukai}.
This also provides an independent proof of the classification in \cite{DKeum} of possible finite groups $G$
 acting tamely symplectically  on $X_{0,p}$ and more generally on any supersingular K3 surface in characteristic $p$,
but now supplemented by the existence decision  which makes for the second main novelty of Theorem \ref{thm}.

\begin{theorem}
\label{thm:tame}
A finite group $G$ admits a tame action on some supersingular K3 surface $X$ in characteristic $p$
with $\NS(X)^G$ of rank two if and only if
$p\nmid |G|, ~X\cong X_{0,p}$ and
$G$ can be realized as a subgroup of $M_{23}$ acting with 4 orbits on $\{1,\hdots,24\}$
such that the orbit lengths $l_1,\hdots,l_4$ satisfy
\begin{eqnarray}
\label{eq:Legendre-tame}
\left(\frac{l_1\cdots l_4}p\right) = -1.
\end{eqnarray}
\end{theorem}

\begin{proof}
{
Let $G$ act tamely symplectically on a supersingular K3 surface $X$ in characteristic $p$
such that $\rank\NS(X)^G=2$. In order to apply Proposition \ref{prop:tame},
we have to show that $G\subset O(\Lambda)$.
To see this, it suffices to prove that $\sigma_0(X)=1$
and apply Theorem \ref{cor:rk<=3}.
This problem serves as an incentive for a little detour which also shows
that the two Artin invariant conditions from \cite[Thm 4.7]{DKeum} (featuring below as $\sigma_0(X)=1$ and $\sigma_0(Y)=1$) 
are in fact equivalent:

\begin{proposition}
\label{prop:s_vs_s}
Let a group $G$ act tamely symplectically on a supersingular K3 surface $X$
such that $\rank(\NS(X)^G)\leq 5$. Denote the minimal resolution of $X/G$ by $Y$.
Then $Y$ is also supersingular, and $\sigma_0(X)=\sigma_0(Y)=1$.
\end{proposition}

\begin{proof}
%For brevity, we only elaborate the proof for $\rank(\NS(X)^G)=4$,
%but the rank 5 case is completely analogous.
%
For starters, the supersingularity of $Y$ classically follows 
from Shioda's theorem on the Lefschetz number \cite[Prop.\ 5]{Shioda-Alg}.
Precisely, $Y$ arises from the quotient $X/G$ by resolving some ADE-singularities
(hence is indeed K3) whose exceptional curves generate a root lattice $R\subset\NS(Y)$ of rank $\geq 19\;  (= \rank(\NS(X)_G)$).
By construction, $p\nmid\det(R)$ whence the isomorphism of $p$-Sylow-groups $A(\NS(Y))_p\cong A(R^\perp)_p$ follows.
But then, as usual, 
\begin{eqnarray}
\label{eq:R}
2\sigma_0(Y) = l_p(\NS(Y)) = l_p(R^\perp) \leq \rank(R^\perp) \leq 3,
\end{eqnarray}
which gives the claim $\sigma_0(Y)=1$.

To see that also $\sigma_0(X)=1$, we can appeal to the general argument from \cite[Lemma 4.1]{WZ}
to infer that $\NS(X)_G$ is $p$-unimodular.
Hence the above argument for $Y$ literally carries over to $X$ to show that 
$$2\sigma_0(X) = l_p(\NS(X)^G) \leq \rank (\NS(X)^G)
$$
which presently implies $\sigma_0(X)=1$ as claimed.
\end{proof}

\begin{remark}
It follows more generally from the birational invariance of the Lefschetz number \cite{Shioda-Alg} 
that any two K3 surfaces $X,Y$ admitting a  dominant rational map 
$X\dasharrow Y$ with separable Galois closure share the same Picard number.
In fact, if $\rho(X)=\rho(Y)<22$ and the degree of the Galois closure is relatively prime to the characteristic, 
then the heights  also agree by \cite[Lem.\ 5.3]{KS-Zariski}.
T.\ Katsura informed us that there is  a generalization of 
the last statement of Proposition \ref{prop:s_vs_s} 
using the a-number defined in \cite[Def.\ 2.1]{vdGK}.
Namely, the a-number is preserved in the above setting in odd characteristic by \cite[Lem.\ 5.4]{KS-Zariski},
and Katsura can show for a supersingular K3 surface $X$ in odd characteristic that $a(X)=1 \Leftrightarrow \sigma_0(X)=1$.
\end{remark}

We return to the proof of Theorem \ref{thm:tame}.
With Proposition \ref{prop:s_vs_s} in place, we infer from Theorem \ref{cor:rk<=3}
that $G\subset O(\Lambda)$ with $\Lambda^G$ of rank 4,
so we can apply Proposition \ref{prop:tame}.
%we have $\sigma_0(X)=1$ and
%the group $G$ admits an action on $\Lambda$ with $\Lambda^G$ of rank 4.
With the conditions from Proposition \ref{prop:tame}, }
the tables of \cite{HM290} return  exactly the 10 maximal groups listed in Table \ref{tab} 
(confirming \cite[Thm.\ 5.2 and Rem.\ 5.3 (3)]{DKeum}).
Indeed, for all other groups in \cite{HM290} with rank$(\Lambda^G)=4$,
the fixed lattice $\Lambda^G$ has either length greater than 2 (not satisfying Proposition \ref{prop:tame} (ii))
or determinant a perfect square (incompatible with Proposition \ref{prop:tame} (iv)).

\begin{table}[ht!]
$$
\begin{array}{|cccccccc|}
\hline
\text{no.} & G &\;\;\;& |G| &\;\;\;& \text{genus symbol} & \;\;\;& l_1, l_2, l_3, l_4 \\
\hline
102 & M_{21} && 2^7.3^2.5.7&&
2_\II^{-2}3^{-1}7^{-1} && 1,1,1,21
\\

106 & 2^4:\mathfrak A_6 && 2^7.3^2.5 && 4_5^{-1} 8_1^{+1}3^{+1}
&& 1,1,6,16
\\
108 & \mathfrak A_7 && 2^3.3^2.5.7&&
3^{+1}5^{+1}7^{+1} && 1,1,7,15
\\
110 & M_{20}:2 && 2^7.3.5&&
4_3^{-1}8_1^{+1}5^{-1} && 
\begin{matrix}
1,1,2,20\\
1,2,5,16
\end{matrix}
\\
111 & 2^3:L_2(7) && 2^6.3.7&&
4_2^{+2}7^{+1}
 &&
\begin{matrix}
1,1,8,14\\
1,7,8,8
\end{matrix}
\\
112 & {2^2}.\mathfrak A_{4,4} && 2^7.3^2&&
8_6^{-2}3^{-1} &&
1,3,4,16
\\
118 & \mathfrak S_6 && 2^4.3^2.5 && 2_\II^{-2}3^{+2}5^{+1} && 1,2,6,15
\\
119 & M_{10} && 2^4.3^2.5 && 2_5^{+1}4_1^{+1}3^{-1}5^{+1} && 1 ,1, 10,12
\\
121 & 2^4:\mathfrak S_{3,3} && 2^6.3^2&&
4_7^{+1}8_1^{+1}3^{+2} &&
1,3,8,12\\

134 & 3^2:SD_{16} && 2^4.3^2&&
2_1^{+1}4_1^{+1}3^{-1}9^{-1} && 1,2,9,12
\\
\hline
\end{array}
$$
\caption{Maximal tame groups with 4 orbits}
\label{tab}
\end{table}

Table 1 in \cite{HM290} also indicates that all these groups are subgroups of $M_{23}$ with 4 orbits
(so the claim holds true also for their 17 subgroups with 4 orbits as classified in \cite[Thm.\ 5.2]{DKeum}).
By inspection, the product of orbit lengths is in the same square class as $\det\Lambda^G$
for the given $G$-action on $\Lambda$,
so condition (iv) from Proposition \ref{prop:tame}
translates exactly into \eqref{eq:Legendre-tame}.
This confirms  the if-direction of the  theorem for all tame maximal groups from \cite{HM290}.

{
It remains to consider the actions of non-maximal subgroups $G'$ of $M_{23}$ with 4 orbits on $\Lambda$.
Clearly they have the same orbits lengths and the same fixed lattice $\Lambda^{G'}$
as the maximal groups $G$ with 4 orbits extending their action.
Hence \eqref{eq:Legendre-tame} holds true automatically unless $G$ is wild while $G'$ is tame.
This, however, may only happen in 3 special cases:
$$
\begin{array}{c|c|c}
G' & G & p\\
\hline
2^4:(3^2:4) & 2^4:\mathfrak A_6 & 5\\
2^4:(5:4) & M_{20}:2 & 3\\
2^3:7 & 2^3:L_2(7) & 3
\end{array}
$$
In each case, $p\nmid\det(\Lambda^{G'})=\det(\Lambda^G)$,
and same for the product of the orbit lengths,
so our previous reasoning goes through for $G'$ to derive \eqref{eq:Legendre-tame}
(cf.\ Remark \ref{rem:tame-wild}). 
This completes the proof of the if-direction of Theorem \ref{thm:tame}.
%
%which is a perfect square modulo $p$ (by inspection of the orbit lengths in Table \ref{tab})

 The converse direction is quite easy.
 All groups from Table \ref{tab} satisfy conditions (ii), (v) of Proposition \ref{prop:tame},
 hence we can apply the crystalline Torelli Theorem \ref{thm:cTorelli}
  to derive the symplectic $G$-action on $X_{0,p}$ 
 in all characteristics $p$ satisfying $p\nmid |G|$ and \eqref{eq:Legendre-tame}.
 Here we also use that tameness implies $p>3$ by inspection of Table \ref{tab}.
 At the same time, this settles all subgroups of the maximal groups,
 except possibly for the above special cases where $G'$ is tame, but $G$ is not.
There, however, $l_1l_2l_3l_4$ is a square modulo $p$, so there is nothing left to prove. 
 }
 \end{proof}
%
%
%Of course, if some group $G$ can be realized, then so can all its subgroups $G'$.
%However, it remains to cover the case where the maximal group $G$ is wild in a given characteristic $p$
%while some subgroup $G'$ with $\Lambda_{G'}=\Lambda_G$ is tame.
%By Lemma \ref{lem:aux}, this implies that $p\nmid\det\Lambda^{G'} = \det\Lambda^G$,
%so Proposition \ref{prop:tame} gives the same existence conditions for $G$ and $G'$ 
%(with the adjustment to the wild case from Remark \ref{rem:tame-wild}).
%With $G$, the subgroup $G'$ is also a subgroup of $M_{23}$, 
%so it remains to verify the translation of the determinant condition
%into the product of orbit lengths from Theorem \ref{thm:tame}.
%But then a close inspection of Table \ref{tab} learns us that for $p\mid |G|$ we have
%\[
%p\mid\det\Lambda^G \Longleftrightarrow p\mid l_1\hdots l_4,
%\]
%so Theorem \ref{thm:tame} also holds for all non-maximal groups $G$ alike.
%
%
%We continue by checking in the tables of \cite{HM290}
%whether the conditions from Proposition \ref{prop:tame} are satisfied.
%This turns out to be true for all 10 eligible (maximal) groups listed in Table \ref{tab}.
%The existence of the symplectic $G$-action on the appropriate $X_{0,p}$
%then follows from Theorem \ref{thm:cTorelli}
%(since $p\neq 2,3$ by the tameness assumption).

\begin{remark}
Due to the determinant condition,
symplectic actions of the groups from Table \ref{tab} can be realized in exactly half of the characteristics,
but there is one subgroup with 4 orbits, namely $O_{48}$,
which has two different actions on $\{1,\hdots,24\}$ with  orbit lengths
$1,1,6,16$ resp.\ $1,3,4,16$
such that their respective products lie in two different square classes.
The group $O_{48}$, which can be represented as a subgroup of both $2^4:\mathfrak A_6$ and $ {2^2}.\mathfrak A_{4,4}$,
can thus  be realized in 75\% of all characteristics.
\end{remark}

\begin{remark}
The groups in Table \ref{tab} were determined in \cite{DKeum}
as the maximal possible tame finite symplectic automorphism groups of (supersingular) K3 surfaces.
Existence, however, was left open.
Our result covers both aspects;
in addition, our approach offers the additional advantage of being completely independent of 
the arguments in \cite{DKeum}.

We take this opportunity to note two corrections to \cite{DKeum}.
For the group $\#121$ in Table \ref{tab},
\cite[Thm.\ 5.2]{DKeum} states an isomorphism 
$2^4:\mathfrak S_{3,3}\cong\mathfrak S_{4,4}\cong \mathfrak S_4 \times\mathfrak S_4$.
This cannot be correct since the group on the right hand side contains elements of order 12, but $M_{23}$ does not
by Table \ref{tab:23}.

In \cite[Rem.\ 5.3 (1)]{DKeum},
it is stated that $M_{23}$ admits another subgroup, namely $M_{20}':2$,
which does not act tamely symplectically on any supersingular K3 surface.
However, with orbit lengths $(1,2,5,16)$ as stated in loc.\ cit.,
this would necessarily act symplectically on $X_{0,p}$ whenever $(10/p)=-1$,
for instance by the geometric construction explained in the next remark.
Indeed,
the Mathieu character applied to $M_{20}':2$ does not return an integer,
so this group cannot be a subgroup of $M_{23}$.
\end{remark}

\subsection{Alternative geometric approach}

There is an intuitive geometric reasoning behind Proposition \ref{prop:tame} and Theorem \ref{thm:tame}
which combines the lines of argument of Kond\=o--Mukai \cite{Kondo-Mukai} 
and of Xiao \cite{Xiao}.

Namely, one can pursue a case-by-case analysis and
consider the $G$-quotients of $X_{0,p}$ for all $G$ in Table \ref{tab}.
By Proposition \ref{prop:s_vs_s},
their minimal resolutions $Y$ are again isomorphic to $X_{0,p}$.
Then we can rule out all those characteristics
not satisfying \eqref{eq:Legendre-tame}
by checking that the (primitive closure of the) 
root lattice $R$ spanned by the exceptional curves on the minimal resolution of the quotient 
(uniquely determined using \cite[Thm.\ 5.2]{DKeum} and, for the saturation of $R$, \cite{S-nodal})
does not embed  into $\NS(X_{0,p})$.

On the other hand, one can
use the property that $G\subset M_{23}$ to endow the Niemeier lattice $N=N(A_1^{24})$
with a $G$-action 
to exhibit the rootfree lattice  $N_G$ encoded in the orbit lengths;
then check that $N_G$ embeds into $\NS(X_{0,p})$
for those characteristics
 satisfying \eqref{eq:Legendre-tame}
 and apply Theorem \ref{thm:cTorelli} to exhibit the symplectic $G$-action on $X_{0,p}$.

Next to the tedious case-by-case analysis, there are two main drawbacks of this approach.
Firstly, there is no conceptual argument for condition \eqref{eq:Legendre-tame}
(which, in fact, only arises a posteriori).
Secondly, this approach does not lend itself to the wild case
as the quotients need not be K3 in that case.

\section{Classification in rank 4 -- wild case}
\label{s:rk4=wild}

We continue with the wild case in rank 4.
Recall that, by \cite[Thm.\ 2.1]{DKeum}, this reduces the possible characteristics generally to $p\leq 11$.
For $X_{0,p}$, this can also be inferred independently from the classification \cite{HM290}
as noted in Remark \ref{rem:p11}.

\subsection{Wild group actions in characteristic $2$}

We start by ruling out the bulk of $G$-actions with $l_2(\Lambda_G)>\rank \Lambda^G-2$ where
Lemma \ref{lem:p_vs_rk} determines the characteristic
 to be $p=2$.
 As lattices with $v(2)\not\hookrightarrow A(\Lambda_G)$ have already been excluded in  \ref{ss:p=2,rk=3},
 we can glue $\Lambda_G$ to $H^{(2)}$ along $v(2)$ and use Lemma \ref{lem:Witt}
 to determine the genus symbol of the saturated lattice $L' = (\Lambda_G\oplus H^{(2)})'\subset L_{1,25}$.
 Whenever the 2-constituent of $A(\Lambda_G)$ is even, 
 this excludes sign walking between this costituent and the unimodular one;
 then the determinant condition as in Lemma \ref{lem:det} gives an immediate contradiction for no.s
 \begin{eqnarray*}
 8, 11, 12, 14, 16, 25, 27, 32 - 34, 36, 40, 42, 49, 50, 54,  65 - 67, 71, 75,\\
 79,  81, 83, 88 - 90, 
 99, 100, 103, 105, 113, 115, 125, 126, 135,  145, 146.
 \end{eqnarray*}
As an illustration, we give details for the cases where the genus symbol given in \cite{HM290}
has to be corrected. As before, $d'$ denotes the $2$-primary part of $-\det(L')$ (or, equivalently, of $-\det(\Lambda_G)$).

\begin{table}[ht!]
$$
\begin{array}{c|c|c|c|c}
\text{no.} & G & \text{genus symbol of } \Lambda_G & \text{genus symbol of } L' & d' \\
\hline
81 & [2^7] & 2_\II^{+2} 4_6^{+2} 8_7^{+1} & 4_2^{-2} 8_7^{+1} & 1 \\
103 & [2^{12}3] &
2_\II^{+2} 4_3^{-1}8_5^{-1} & 4_7^{+1}8_5^{-1} & -1\\
125 & [2^73] & 2_\II^{-2} 4_7^{+1} 8_7^{+1} 3^{-1} & 4_7^{+1} 8_7^{+1} 3^{-1} & -3\\
145 & [2^43] & 2_\II^{-2} 4_3^{-1} 8_1^{+1} & 4_3^{-1} 8_1^{+1} & -1
\end{array}
$$
\caption{Some groups $G$ with $v(2)\hookrightarrow A(\Lambda_G)$}
\label{table5}
\end{table}

\vspace{-.5cm}

\subsection{Wild group actions in odd characteristic}

There are a few group actions in rank 4
where Lemma \ref{lem:p_vs_rk}  determines the characteristic
$p$ to be odd.
Having excluded a few of them in Section \ref{s:cond},
the remaining three groups 
(not contained in $M_{23}$) 
are covered  below.

\subsubsection*{No.s $101, 109$} 
These groups are subgroups of $U_4(3)$ and thus occur in characteristic $p=3$.

\subsubsection*{No.\ $122$} 
This group is a subgroup of $U_3(5)$ and thus occurs
in characteristic $p=5$.

\subsection{Wild group actions with square determinant}
\label{ss:square}

The considerations in Section \ref{s:rk4=tame} excluded all those tame group actions in rank 4
such that $\det \Lambda^G$ is a perfect square
(by Proposition \ref{prop:tame} (iii)).
Here we verify which actions  are compatible with the wild setting.
It turns out that each group action can be realized 
on the supersingular K3 surface of Artin invariant 1 in exactly one characteristic $p$
(indicated in the last column of Table \ref{table6}).

To see this, we make use of Remark \ref{rem:tame-wild}
and apply Proposition \ref{prop:tame} (iii) whenever $p'\nmid\det\Lambda^G$.
Other wild primes can be excluded by the determinant condition applied to a putative invariant lattice $\NS^G$
(whose determinant agrees with that of $\Lambda_G$ up to sign and squares,
so $d'$ denote the $p'$-primary part of $-\det \Lambda^G$ as usual).

 \begin{table}[ht!]
$$
\begin{array}{c|c|c|c|c|c}
\text{no.} & G & \text{genus symbol of } \Lambda_G & p'\nmid\det\Lambda_G & (d'/p')\neq \varepsilon_{p'} & p  \\
\hline
120 & L_2(11) & 11^{+2} & 2,3,5 & - & 11\\
128 & (3\times A_5):2 & 3^{-2} 5^{-2} & 2 & 3 & 5\\
129 &  L_2(7) \times 2 & 2_\II^{+2} 7^{+2} & 3 & 2 & 7
\end{array}
$$
\caption{Groups $G$ in rank 4 with square determinant}
\label{table6}
\end{table}

\begin{summary}
\label{summary}
The supersingular K3 surface $X_{0,p}$ admits a wild symplectic action by $G$ exactly for the prime $p$ 
indicated in the last column of Table \ref{table6}.
The geometric realizations
follow from Theorem \ref{thm:cTorelli},
but they can also be exhibited explicitly, see \ref{ss:explicit-rk4}.
Each of these three groups is contained in a maximal group from Table \ref{table1},
namely in $M_{22}, \mathfrak A_8$ resp.\ $M_{21}.2_2$.
Note that this is in perfect agreement with Theorem \ref{thm1}.
\end{summary}

\begin{remark}
It can be shown using Theorem \ref{thm:cTorelli}
that the symplectic actions from Summary \ref{summary} extend over any supersingular K3 surface
of Artin invariant $\sigma_0\leq 2$ in the specified characteristic $p$.
\end{remark}

\subsection{Groups from the tame rank 4 case}

We conclude this section by discussing the 10 groups from Table \ref{tab}
for the primes $p\leq 11$, especially in the wild setting.

\begin{proposition}
\label{prop:max}
A group from Table \ref{tab} is realized in characteristic $p\leq 11$
if and only if it is contained in a maximal group from the rank 3 case,
with the exceptions of
$2^2.\mathfrak A_{4,4}$ (no.\ 112) in characteristic $3$ and
 $2^4:\mathfrak S_{3,3}$ (no.\ 121) in characteristics $3$ and $11$.
\end{proposition}

%In detail, we determine whether they can be realized in characteristic $p$,
%and if so, whether they are contained in a maximal group from the rank 3 case 
%(so as to confirm Theorem \ref{thm1}).

\begin{proof}
Of course, existence can be derived as a consequence of any subgroup relation with a group from the rank 3 setting.
Here, next to the tables of \cite{Atlas}, the supplementary files available with the paper
are very useful, but one can also use a computer algebra system, notably GAP \cite{GAP}.
For the  two cases with  $G=2^4:\mathfrak S_{3,3}$,
existence follows in characteristic $p=11$ from Theorem \ref{thm:tame},
and in characteristic $p=3$ from the explicit construction in \ref{sss:121}.
For the remaining group $2^2.\mathfrak A_{4,4}$ in characteristic $3$,
the same is achieved in \ref{sss:112}.

%gluing $\Lambda_G$ to the lattice $M'=\langle 4\rangle\oplus\langle-8\rangle$
%by appealing to the crystalline Torelli theorem from \cite{Ogus79} and \cite{BL19}.
%\marginpar{check subgroup or replace by explicit realisation to avoid reference to cryst. Torelli?!}

For non-existence, we apply the determinant criterion from Proposition \ref{prop:tame},
using Remark \ref{rem:tame-wild} whenever possible.
The remaining cases, indicated by asterisks in Table \ref{table7}, will be explained below.

 \begin{table}[ht!]
$$
\begin{array}{c|c|c||c|c||c|c}
\text{no.} & G & \text{gen. symb } \Lambda_G & \text{ex.} & \text{subgroup of} & \text{non-ex.} &\text{reason}  \\
\hline\hline
102 & M_{21} & 2_\II^{-2}3^{-1}7^{-1}& 2 & M_{21}.2_1
& 5 & (21/5)=1
\\
&&& 3 & U_4(3) &&\\
&&& 7 & M_{21}.2_2 &&\\
&&& 11 & M_{22} &&\\
\hline
106 & 2^4:\mathfrak A_6 & 4_5^{-1} 8_1^{+1}3^{+1}
& 3 & U_4(3) & 2 & *
\\
&&& 7 & 2^4:\mathfrak A_7 & 5 & (6/5)=1\\
&&& 11 & M_{22} &&
\\
\hline
108 & \mathfrak A_7 &
3^{+1}5^{+1}7^{+1} & 3 & U_4(3) & 2 & (105/2)=1\\
&&& 5 & \mathfrak A_8 &&\\
&&& 7 & 2^4:\mathfrak A_7&&\\
&&& 11 & M_{22} &&
\\
\hline
110 & M_{20}:2 &
4_3^{-1}8_1^{+1}5^{-1} &
5 & 2^4 :(3\times\mathfrak A_5):2 & 2 & *\\
&&& 7 & 2^4:\mathfrak A_7
 & 3 & (10/3)=1\\
&&& 11 & M_{22}&&
\\
\hline
111 & 2^3:L_2(7) &
4_2^{+2}7^{+1}
 &
 5 & \mathfrak A_8 & 2 & *\\
 &&& 7 & 2^4:\mathfrak A_7& 3 & (7/3)=1\\
 &&& 11 & M_{22}&&\\
 \hline
112 & {2^2}.\mathfrak A_{4,4} &
8_6^{-2}3^{-1} &
3 & - & 2 & **\\
&&& 5 & 2^4 :(3\times\mathfrak A_5):2 & 11 & (3/11) = 1\\
&&& 7 & 2^4:\mathfrak A_7&&\\
\hline
118 & \mathfrak S_6 & 2_\II^{-2}3^{+2}5^{+1} &
2 & \Aut(\mathfrak S_6) & 11 & (5/11)=1 \\
&&& 3 & U_4(3) &&\\
&&& 5 & \mathfrak A_8&&\\
&&& 7 & M_{21}.2_2&&\\
\hline
119 & M_{10} & 2_5^{+1}4_1^{+1}3^{-1}5^{+1} &
2 & M_{21}.2_1 & 7 & (30/7)=1\\
&&& 3 & U_4(3)&&\\
&&& 5 & U_3(5) &&\\
&&& 11 & M_{22}&&\\
\hline
121 & 2^4:\mathfrak S_{3,3} &
4_7^{+1}8_1^{+1}3^{+2} &
3 & - & 2 & *\\
&&& 5 & \mathfrak A_8 & 7 & (2/7)=1\\
&&& 11 & - &&\\
\hline
134 & 3^2:SD_{16} &
2_1^{+1}4_1^{+1}3^{-1}9^{-1} &
2 & \Aut(\mathfrak S_6) & 3 & ***\\
&&&7 & M_{21}.2_2 & 5 & (6/5)=1\\
&&& 11 & M_{11} &&
\end{array}
$$
\caption{Groups $G$ in rank 4 in small characteristic}
\label{table7}
\end{table}
%
%The given reasons refer to the determinant criterion from Proposition \ref{prop:tame},
%using Remark \ref{rem:tame-wild}. In what follows, we comment 

\subsubsection{Reason*}
\label{sss1}

{ 
Consider the embedding $\Lambda_G\oplus H^{(2)}\hookrightarrow L_{1,25}$
with orthogonal complement $\NS(X_{0,2})^G$ of rank $2$ as in Figure \ref{fig:1,25}.
For length reasons, $\Lambda_G$ and $H^{(2)}$ thus have to glue.
By Lemma \ref{lem:odd-const}, however, $A(H^{(2)})\cong v(2)\not\hookrightarrow A(\Lambda_G)$,}
so these two discriminant groups can only glue along a single odd integral 2-torsion vector.
In each case, the glue vector in  $A(\Lambda_G)$ can be taken to be any element $2v$ for $v$ an odd 4-torsion element.
Writing $v(2) = \langle e,f\rangle$, we thus postulate that $e+2v\in L' = (\Lambda_G\oplus H^{(2)})'\subset L_{1,25}$.
But then $w=v+f$ generates $A(L')$ together with the other orthogonal generators of $A(\Lambda_G)$,
and with $w^2 = v^2 + 1$ we obtain the genus symbol of $L'$
with the following properties:
\begin{itemize}
\item
$l_2(L') = l_2(\Lambda_G)=2$;
\item
the $2$-adic symbol has no constituent at $2$ and
\item 
exactly one sign switched compared to the genus symbol of $\Lambda_G$;
\item
the determinant lies in the same square class as $\det(\Lambda_G)$.
\end{itemize}
The first two properties imply that $l_2(\NS(X_{0,2})^G)=2$,
hence  the analogue of the determinant condition from Lemma \ref{lem:det} applies to $L'\hookrightarrow L_{1,25}$.
%(where, of course, $L'^G=\NS(X_{0,2})^G$ as in Figure \ref{fig:1,25}).
But then the last two properties give the desired contradiction in each case.

\subsubsection{Reason**}
\label{sss2}

All 2-torsion elements in $A(\Lambda_G)$ are isotropic,
so there cannot be any glue with $H^{(2)}$,
and $L' = \Lambda_G \oplus H^{(2)}\subset L_{1,25}$ attains length $4$, 
exceeding the rank of its purported orthogonal complement $\NS^G$  inside $L_{1,25}$, contradiction.
%(i.e.\ the corank in $L_{1,25}$ if there were an embedding which is thus excluded).

\subsubsection{Reason***}
\label{sss3}

By considering the values attained by $A(\Lambda_G)$ on the 3-torsion elements,
we detect that  $\Lambda_G$ can only glue  to $H^{(3)}$ along a non-isotropic 3-torsion element of the respective discriminant group.
The resulting lattice $L' = (\Lambda_G \oplus H^{(3)})'\subset L_{1,25}$ turns out to have
genus symbol $2_1^{+1}4_1^{+1}3^{+1}9^{-1}$,
so the determinant condition at $3$ gives the contradiction $(-8/3)=1$.
\end{proof}

\subsection{Concluding remark in rank 4}

For all the remaining $G$-actions in rank $4$ (no.s 124, 133, 137, 141, 149, all not contained in $M_{23}$),
the characteristic is $p=2$ by Lemma \ref{lem:p_vs_rk};
existence thus follows
since the groups are contained in $M_{21}.2_1$ or $\Aut(\mathfrak S_6)$.
This is easy for all groups except for 
no.\ 141 where the supplementary files of \cite{HM290} provide an inclusion  $G\subset\mathfrak A_6.2$.
By \cite{Atlas}, there are 3 such extensions ($\mathfrak S_6, M_{10},$ PGU$_2(9)$), 
but each is contained in $\mathfrak A_6.2^2 = M_{10}.2 = \Aut(\mathfrak S_6)$ confirming the claim.
In fact, using GAP, one can verify that the group $G$ is isomorphic to $\mathfrak S_{3,3}:2$ and thus a normal subgroup
of $3^2:SD_{16}$, one of the maximal groups in rank 4.
%\marginpar{I'm sure that $G$ also has rank 4, i.e. (ix) in \cite[Thm.\ 5.2]{DKeum},but I'm less sure about 
%the isomorphism with $3^2:8$ given there...}
%\marginpar{I'm still puzzled about the ranks for \#141 -- so with the $G$-action on $\Lambda$, this is nor realized because 
%it's only a subgroup of some $\mathfrak A_6.2$ in rank 3, but abstractly it is realized as a subgroup of $ \Aut(\mathfrak S_6)$?!
%Then it has to feature as a non-maximal group in some rk as well?!?!}

In consequence, each of these groups acts symplectically  on the supersingular K3 surface $X_{0,2}$.
%though we refrain from deciding about the very given $G$-action.

%
%
%
%But then how about no.\ 141: $G=[2^33^2] (\# 40)$ which seems to act with $\NS^G = 6 + (-18)$
%not containing elements of square 14 or 30, so the action can not be induced from a maximal one 
%(compatible with [HM]).

\section{Rank $\geq 5$ case}
\label{s:rk5}

With the maximal groups, all subgroups are automatically realized.
Therefore it will not be not hard to show that all subgroups of $M_{23}$ with at least 5 orbits 
 are in fact realized on $X_{0,p}$ in odd characteristic.
%as the only open cases concern characteristic $2$ by Lemma \ref{lem:p_vs_rk}.
For completeness and future reference, we record the precise statement for 
Mukai's 11 maximal groups from the complex setting \cite{Mukai}:
\begin{eqnarray}
\label{eq:max}
\;\;\;\;\;\;\;
L_2(7),\; \mathfrak A_6,\; \mathfrak S_5,\; M_{20},\; F_{384},\; \mathfrak A_{4,4},\;
 T_{192}, \;H_{192}, \;N_{72}, \;M_9, \;T_{48}.
\end{eqnarray}

\begin{proposition}
\label{prop:Mukai}
\begin{enumerate}
\item If $p>2$ is prime, then all groups from \eqref{eq:max} act symplectically 
on some K3 surface in characteristic $p$;
\item
more precisely, the symplectic action can be realized
on $X_{0,p}$, and, 
\item 
for each given $2<p\leq 11$,   
each group from \eqref{eq:max} is contained in one of the maximal rank 3 groups (listed in Theorem \ref{thm1}).
\item
The groups $L_2(7), \mathfrak A_6, \mathfrak S_5, M_{20}, N_{72}, M_9$ 
act symplectically on $X_{0,2}$ as they are all contained in $M_{21}.2_1$ or $\Aut(\mathfrak S_6)$.
\item
The groups
 $F_{384}, \mathfrak A_{4,4}, T_{192}, H_{192}, T_{48}$
 cannot act symplectically on $X_{0,2}$.
 %any supersingular K3 surface in characteristic $2$.
 \end{enumerate}
\end{proposition}

\begin{proof}
Statements (3) and (4) are easy to check  using \cite{Atlas} or the auxiliary files of \cite{HM290};
 this proves (1) and (2) at once for all odd primes $p\leq 11$.

To realize the full scope of (1) (for the remaining primes, but also for many groups for $p\leq 11$), 
one can argue directly with the equations given in \cite{Mukai},
noting that 
\begin{itemize}
\item
If $p\nmid |G|$, then the model has good reduction at $p$.
\item
If $p^2\nmid |G|$, then the model has good reduction or
acquires isolated rational double point singularities upon reducing modulo $p$,
so the minimal resolution is a K3 surface.
\end{itemize}

At a tame prime $p>11$,
we have $p\nmid\det \Lambda_G$ by Lemma \ref{lem:aux},
so (2) can be proved similarly to Proposition \ref{prop:tame}.
Namely $\Lambda_G$ glues into the K3 lattice $\Lambda_\text{K3}$
with orthogonal complement $T$ of signature $(3,0)$.
We derive an auxiliary lattice $M'$ of signature $(1,2)$
by postulating the same genus symbol as for $T$, but with $p^{\varepsilon 2p}$ added.
As in the proof of Proposition \ref{prop:tame}, the discrepancies in signature and $p$-excess
balance out, and evenness of $M'$ follows accordingly.
Hence $\Lambda_G$ glues into $\NS(X_{0,p})$,
and the symplectic action on $X_{0,p}$ follows from Theorem \ref{thm:cTorelli},
proving (2) for all odd primes.

As for (5), we 
%first reduce to $X_{0,2}$ using Proposition \ref{prop:reduction} and then 
consider the single cases.
 First we rule out the $G$-actions on $\Lambda$ from \cite{HM290}.
 To start with, 
$A(T_{48})$ does not glue at all to $v(2)$,
so we get the same contradiction as in \ref{sss2}.

For $G=\mathfrak A_{4,4}$,  
the coinvariant lattice $\Lambda_G$ glues to $H^{(2)}$ along a single vector of the discriminant groups (by Lemma \ref{lem:odd-const}),
and $L'=(\Lambda_G\oplus H^{(2)})'\subset L_{1,25}$ has genus symbol $2_\II^{-2} 8_1^{+1} 3 ^{+2}$,
so the determinant condition $(9/2)=1$ gives a contradiction for its putative orthogonal complement.

The three remaining group actions have no 2-constituent, so they can be treated as in \ref{sss1}, leading to the same kind of contradiction
based on the determinant condition. 
This excludes the actions on $\Lambda$ from \cite{HM290}.

It remains to check that the five groups from (5) are not abstract subgroups of the maximal groups in characteristic $2$,
i.e.\ of $M_{21}.2_1$ or
$\Aut(\mathfrak{S}_6)$ (since the only other remaining groups, to be excluded in \ref{ss:remain},
are out of question for size reasons).
This amounts to a  case-by-care analysis using \cite{Atlas}, \cite{HM290} and GAP \cite{GAP}
(and the GAP identifiers listed in \cite[Table 10.2]{Hashi})
As an illustration, we sketch the case  $G=T_{48}$.
With GAP, one can directly verify that $G\not\subset \Aut(\mathfrak{S}_6)$.
Assuming that $G\subset M_{21}.2_1$, it is convenient to reduce the complexity as follows;
namely, since $5\mid\det(\Lambda^{M_{21}.2_1})$, 
Lemma \ref{lem:aux} implies that $\Lambda^G\subsetneq \Lambda^{M_{21}.2_1}$.
Hence $G\subset G'$ for some subgroup $G'\subset \Lambda^{M_{21}.2_1}$ such that $\rank(\Lambda^{G'})\geq 4$.
By the auxiliary files of \cite{HM290}, there are 5 maximal such groups
of which $\mathfrak S_5$ cannot contain $G$ for size reasons.
For the other 4, one can either use GAP directly or, for $M_{21}$ and $M_{10}$, reduce further as above.

The arguments for the groups $F_{384}, \mathfrak A_{4,4}, T_{192}, H_{192}$ are similar and thus omitted for brevity.
\end{proof}

\subsection{Remaining group actions}
\label{ss:remain}

Considering all subgroup relations among the groups acting on $\Lambda$ with $\rank \Lambda^G \geq 3$,
and groups excluded so far,
there are only  four $G$-actions uncovered as of yet, namely no.s 28, 39, 41, 43.

Each implies $p=2$ by Lemma \ref{lem:p_vs_rk}
and can be ruled out along the same lines as in the proof of Proposition \ref{prop:Mukai} (5).
We therefore omit the details.

\begin{conclusion}
\label{conclusion:all}
For each group  from \cite{HM290},
we have decided whether it acts symplectically on $X_{0,p}$ in any given characteristic $p$.
\end{conclusion}

For completeness, we record the following summary including the wild setting:

\begin{corollary}
\label{cor:anyM_23}
Any subgroup of $M_{23}$ with at least 5 orbits can be realized symplectically on $X_{0,p}$
with the exception of characteristic $p=2$
where none of
$F_{384}, \mathfrak A_{4,4}, T_{192}, H_{192}, T_{48}, \mathfrak{A}_{4,3},  \Gamma_{25}a_1, \Gamma_5 a_1$
is realized
and same for all subgroups with same orbit decomposition
except for $\Gamma_2d, \Gamma_3e, 4^2.D_6$.
%In particular subgroups of $M_{23}$ with at least 6 orbits can be realized.
\end{corollary}

\begin{proof}
For $p>2$, the statement follows from Proposition \ref{prop:Mukai} (2).
For $p=2$, the eight given subgroups of $M_{23}$ stated explicitly are ruled out in Proposition \ref{prop:Mukai} (5)
and in the first lines of this subsection (the group corresponding to no.\ 41 is not a subgroup of $M_{23}$).

As for their subgroups with same orbit decomposition, 
\cite[Table 10.2]{Hashi} lists 22 of them ($T_{24}$ being contained in both $T_{48}$ and $T_{192}$).
Among them, the three listed groups are contained in $M_{21}.2_1$
(more precisely, in the group no.\ 124 from \cite{HM290}).
For all other groups one can show along the same lines as before 
that none is contained in $\Aut(\mathfrak S_6)$ or $M_{21}.2_1$.
\end{proof}

\subsection{Connection with complex setting}
\label{ss:Mukai}

In fact, we can turn the above arguments around to give an independent proof of Mukai's classification over $\CC$
(Theorem \ref{thm:Mukai}).
It can be derived from \cite{HM290} and our findings based on the following general result:

\begin{theorem}
\label{thm:equiv}
The following are equivalent for a  finite group $G$ 
 and a prime $p$ such that $p\nmid |G|$:
\begin{itemize}
\item
$G$ acts  tamely symplectically on the supersingular K3 surface $X_{0,p}$ % in characteristic $p$
such that $\rank\NS(X_{0,p})_G\leq 19$.
\item
$G$ acts symplectically on some complex K3 surface.
\end{itemize}
\end{theorem}

\begin{proof}
With $G$ acting symplectically, we are comparing the orthogonal complements of $\NS_G$:
$M'$ of signature $(1,r+2)$ inside $\NS(X_{0,p})$ and $T$ of signature $(3,r)$ inside $\Lambda_\text{K3}$
(for some suitable $r$).
It is immediate to adjust the proof of Proposition \ref{prop:Mukai} (2)
to arbitrary rank to go back and forth between $M'$ and $T$.
Then we conclude using the complex or crystalline Torelli theorem in characteristic $p>3$
while the 
existence of the group actions in characteristics $p=2,3$ can be inferred from the explicit models
constructed in the next section (using Theorem \ref{thm1} whose proof does not rely on Theorem \ref{thm:equiv}).
\end{proof}

\section{Explicit realizations}

\label{s:models}

Explicit symplectic actions have been constructed for $U_4(3)$
on the Fermat quartic surface  in \cite{Tate} or \cite{DKeum},
and for $U_3(5)$ on the double cover of $\PP^2$ branched along the Fermat sextic curve
in \cite{DKeum}. 
For more explicit realizations of symplectic group actions in characteristics $5, 7$ and $11$, albeit not maximal, 
see \ref{ss:explicit-rk4}.

\subsection{Maximal subgroups in characteristic $2$}
\label{ss:2}

In \cite{DK}, Dolgachev and Kond\=o exhibit an explicit symplectic action
of the maximal group
$
M_{21}.2_1$
on $X_{0,2}$.
They also claim that the group is the unique maximal finite group
acting symplectically on $X_{0,2}$.
This is in contrast to our result in Theorem \ref{thm1}
which will be substantiated in the next subsection.
Here we comment on the gap in the argument of \cite{DK}.

The uniqueness claim 
relies on the general argument that for a finite group $G$ acting on $X_{0,2}$,
any $G$-invariant ample vector $h$ can be mapped to the  fundamental domain $\mathcal D$ for the action of the Weyl group
(intersected with the positive cone)
by a composition of certain reflections.
Hence $G$ is conjugate to a subgroup of $\Aut(\mathcal D)\cong M_{21}.D_{12}$
(using work of Borcherds).
The authors conclude using the fact 
that $D_{12}$ contains elements corresponding to non-symplectic automorphisms of order $3$
and to the Frobenius endomorphism of $X_0$
(which is, of course, not given by an automorphism of $X_0$).
From this they would like to deduce that, if $G$ acts symplectically, then $G\subset M_{21}.2$.

However, as Mukai points out, the reflections required for mapping $h$ into $\mathcal D$
may involve the induced action of the Frobenius endomorphism.
This facilitates that the class of Frobenius in $\Aut(\mathcal D)$ actually is conjugate to a proper automorphism in $G$
under the reflection in question
-- as seems to happen for $G=\Aut(\mathfrak S_6)$,
the only other possible maximal finite symplectic group by Theorem \ref{thm1}.
%At least this is the only other maximal finite group acting symplectically.

\subsection{Symplectic action by $\Aut(\mathfrak S_6)$}
\label{s:Aut(S_6)}

In this section we focus on the supersingular K3 surface $X_0=X_{0,2}$ of Artin invariant $\sigma=1$ in characteristic $2$.
In \cite[Cor.\ 5.1]{DK} it is stated that any finite  group acting symplectically on $X_0$
necessarily is a subgroup of $M_{21}.2_1$.
Here we provide a counterexample to this claim
(affirming a conjecture from \cite{KK} and, most importantly, supporting Theorem \ref{thm1}
despite the absence of a valid crystalline Torelli theorem in characteristic $2$):

\begin{proposition}
The group $\Aut(\mathfrak S_6)$ admits a symplectic action on $X_0$.
\end{proposition}

\begin{proof}
The action by $\mathfrak S_6$ was already stated in \cite{KK}.
Here we exhibit an independent argument starting from the natural degree 6 surface
\[
Y_0 = \{s_1=s_2=s_3=0\}\subset \PP^5
\]
where $s_i$ denotes the elementary symmetric polynomial of degree $i$ in the homogenous variables of $\PP^5$.
Note that this surface appears prominently in \cite{Mukai}
for the symplectic $\mathfrak A_6$ action over $\CC$
(where $Y_0$ is smooth and thus defines a K3 surface).

In characteristic $2$, $Y_0$ contains 15 isolated singular points 
at the $\mathfrak S_6$-orbit of $[1,1,0,0,0,0]$
which are in fact $A_1$ singularities.
We resolve them to obtain a K3 surface $Y$
to which the $\mathfrak S_6$ action extends from $Y_0$.
Since the group action is defined over $\FF_2$,
it is in fact automatically symplectic.

In addition, $Y_0$ contains (the strict transforms of) 15 lines 
given by
\[
\ell_{jl,mn} = 
x_1+x_i=x_j+x_l=x_m+x_n=0 \;\;\;\; \{i,j,l,m,n\}=\{2,3,4,5,6\}
\]
(which we can equally well identify as $\ell_{1i,jl}$ or $\ell_{1i,mn}$).
Together with the exceptional curves above the 15 nodes,
the lines form a $(15_3)$-configuration
where each line meets three exceptional curves,
and vice versa, each exceptional curve meets three lines.
These 30 smooth rational curves exhibit a nice $\mathfrak S_6$ symmetry
which obviously preserves the set of lines and likewise the set of exceptional curves.
To obtain the claimed action by $\Aut(\mathfrak S_6)$,
it remains to exhibit an involution on $Y$ which acts as a switch between lines and exceptional curves
(and to prove that $Y\cong X_0$).
This can be achieved with the help of a special elliptic fibration 
with two fibres of Kodaira type I$_{10}$.
By the standard theory of elliptic K3 surfaces, the fibration is induced by either fibre,
so we continue to exhibit two disjoint 10-cycles on the graph of lines and exceptional curves.
To this end, we introduce the notation
$E_{ij}$ for the exceptional curve in $Y$ above the node with non-zero coordinates $x_i, x_j$,
so that $E_{ij}$ meets the three lines  
which can be written with a double index $ij$.
Then two disjoint 10-cycles can be given by
\begin{eqnarray*}
F_1 & = & \ell_{12,34}+E_{12}+\ell_{12,36}+E_{45}+\ell_{13,45}+E_{13}+\ell_{13,25}+E_{25}+\ell_{16,34}+E_{34}\\
F_2 & = & \ell_{14,56} + E_{14} + \ell_{14,26} + E_{35} + \ell_{16,24} + E_{24} + \ell_{15,36} + E_{15} + \ell_{15,23} + E_{23}  
\end{eqnarray*}
Note that the 10 remaining lines and exceptional curves form sections of the fibration.
In fact, picking one curve as zero section, say the line $\ell_{12,35}$,
the remaining four lines give $5$-torsion sections by inspection of the height being zero
in the theory of Mordell--Weil lattices \cite{MWL}.
It then follows from the general theory of modular surfaces
that $Y$ is a base change of the elliptic modular surface with a point of order $5$.
By \cite{Kubert}, this is a rational elliptic surface with two fibres of type I$_5$.
Hence the base change has to ramify at these two fibres,
so it is inseparable of degree $2$,
and there are two more reducible fibres of type I$_2$
(comprising some twisted cubics on $Y_0$ defined over $\FF_4$,
but we omit them for space reasons).
In particular, $\rho(Y)=22$ by the Shioda--Tate formula, so $Y$ is supersingular.
Now consider the remaining 5 exceptional curves
which also serve as sections.
Again, by the Shioda--Tate formula, they have finite order,
so we find that $\MW\cong\ZZ/10\ZZ$.
Then the determinant formula \cite[Cor.\ 6.39]{SS} shows that $\NS(Y)$ has
determinant $-4$, i.e.\ $Y$ has Artin invariant $\sigma=1$, 
in agreement with \cite{ES} and \cite{KK}.

Since $Y\cong X_0$ by \cite{Ogus83}, we have exhibited a symplectic $\mathfrak S_6$-action on $X_0$.
Consider the automorphism given by translation by the 2-torsion section $E_{46}$.
(This is the special case of a peculiar torsion section which is not disjoint from the zero section $\ell_{12,35}$;
this may only happen if the characteristic divides the group order.)
By construction, this involution switches lines and exceptional curves while preserving the elliptic fibration structure.
It thus exhibits exactly the outer automorphism of $\mathfrak S_6$ that we need.
\end{proof}

\begin{conclusion}
Both maximal possible finite symplectic groups in characteristic $2$, $M_{21}.2_1$ and $\Aut(\mathfrak S_6)$,
as classified in the course of this paper,
are  supported by explicit geometric realizations on $X_{0,2}$.
\end{conclusion}

\subsection{Maximal subgroups in characteristic $3$}
\label{ss:3}

We have seen in Proposition \ref{prop:max} that next to $U_4(3)$ 
there are two more maximal groups $G$ admitting a 
symplectic action on $X_0 = X_{0,3}$,
the supersingular K3 surface of Artin invariant $\sigma_0=1$ in characteristic $3$,
namely
\[
G = 2^4:\mathfrak S_{3,3} \;\;\; \text{ and } \;\;\; G = 2^2.\mathfrak A_{4,4}.
\]
For lack of a  crystalline Torelli theorem in characteristic $3$ (cf.\ Remark \ref{rem:crys}),
we give an explicit geometric realization for both groups.

\subsubsection{$\boldsymbol{G = 2^4:\mathfrak S_{3,3}}$}
\label{sss:121}

This is a degree 3 extension of the group $G_0=H_{192} = 2^4:D_{12}$ from Theorem \ref{thm:Mukai}.
We shall realize it in characteristic $3$ starting from Mukai's realization of $G_0$ over the complex numbers 
in \cite{Mukai}
via the degree $8$ K3 model given by
\begin{eqnarray}
\label{eq:X3}
X = \left\{
\begin{matrix}
x_1^2+x_3^2+x_5^2 = x_2^2 + x_4^2 + x_6^2\\
~\\
x_1^2+x_4^2 = x_2^2 + x_5^2 = x_3^2 + x_6^2
\end{matrix}
\right\}
\subset\PP^5.
\end{eqnarray}
While this is smooth over $\CC$,
it attains 8 nodes in characteristic $3$ at the points 
\begin{eqnarray}
\label{eq:nodes}
[0,1,0,\pm 1,0,\pm 1], \;\; [1,0,\pm 1, 0,\pm 1, 0].
\end{eqnarray}
In addition, there are 16 lines given by
\[
\{  [\mu, \lambda,\pm \mu, \pm\lambda\pm \mu, \pm\lambda]; [\mu,\lambda]\in\PP^1\}.
\]
One directly verifies that each node is met by exactly 4 lines,
and two nodes are connected by a (unique) line 
if and only if they take the two different shapes in \eqref{eq:nodes}.
On the minimal resolution $\tilde X$, this means that the exceptional divisors and the strict transforms of the lines
form the standard double Kummer configuration of 24 smooth rational curves on a Kummer surface of product type Km$(E\times E')$,
as introduced in \cite[Fig.\ 1]{SI}.
In particular, there are two isotrivial elliptic fibrations
\[
\pi, \; \pi': \;\; \tilde X \to \PP^1
\]
with four fibres of Kodaira type I$_0^*$
(given by lines and the exceptional curves above  four nodes of fixed shape from \eqref{eq:nodes})
and full two-torsion in the Mordell--Weil group (given by the other exceptional curves)
such that the general fibre is isomorphic to $E$ resp.\ $E'$.
(Here the fibrations $\pi, \pi'$ are induced by the projections to each factor of $E\times E'$.)
Together, these 24 smooth rational curves generate a lattice $L\subset\NS(\tilde X)$ of rank $18$ and determinant $-16$,
an index 4 overlattice of $U\oplus D_4^4$, generated by the torsion sections.

\begin{claim}
\label{claim:ss}
$\tilde X$ is supersingular; in particular,
$\tilde X \cong X_0$.
\end{claim}

\begin{proof}
We argue with the complex K3 surface $X_\CC$ 
which has Picard number $\rho(X_\CC)=20$ and transcendental lattice $T(X_\CC) \cong $ diag$(4,12)$
by \cite{Hashi}.
It follows from \cite[\S 3 (I) \& (5.8)]{SM} that $X_\CC \cong$ Km$(E_\CC \times E_\CC)$
where $E$ is the elliptic curve with complex multiplication by $\ZZ[\sqrt{-3}]$ and j-invariant $2^43^35^3$.
Since $E$ has potentially good reduction everywhere,
yielding the elliptic curve $E_0$ with j-invariant $0$ in characteristic $3$,
we infer that $\tilde X \cong$ Km$(E_0\times E_0)$.
This proves both claims.
\end{proof}

The above argument also shows that the generic fibre of $\pi$ (and of $\pi'$)
admits a wild automorphism of order 3, say $\varphi$, which is automatically symplectic.
Explicitly, this can be described by the affine model
\begin{eqnarray}
\label{eq:Km}
(t^3-t) y^2 = (x^3-x).
\end{eqnarray}
This makes the two fibrations to $\PP^1_t, \PP^1_x$ visible,
with singular fibres at the $\FF_3$-points of $\PP^1$ and
2-torsion sections 
$$
P_0 = (0,0), \;\; P_1=(1,0), \;\; P_{-1} =(-1,0)
$$ 
for $\pi$ mapping $\tilde X$ to $\PP^1_t$,
and analogously for $\pi'$ over $\PP^1_x$. 

Each base curves admits the symmetric group $\mathfrak S_4$ as automorphisms,
but the resulting action on $H^{2,0}(\tilde X)$ involves the fourth roots of unity.
However, the subgroup $\mathfrak A_{4,4}$ admits a symplectic action on $\tilde X$
which can be extended  by the switch of fibrations 
$$ 
\varphi_0: \;(t,x,y) \mapsto (x,t,-1/y)
$$
to give a group $G$ of size $576$ acting symplectically.

\begin{proposition}
\label{prop:121}
$G\cong 2^4:\mathfrak S_{3,3}$.
\end{proposition}

\begin{proof}
To prove the proposition, we start by considering the normal subgroup $N\lhd G_0$
which is generated by sign changes in the homogeneous coordinates of \eqref{eq:X3}.
From the action on the nodes, one infers that $N$ corresponds to compositions of the translations 
by 2-torsion sections of the two fibrations $\pi, \pi'$.
In terms of \eqref{eq:Km},
 these translations correspond to the involutions of either base curve $\PP^1_t, \PP^1_x$ 
which preserve no point over $\FF_3$, such as $t\mapsto -1/t$
(each forming the Klein group $V_4$).

The subgroup $D_{12}$ of automorphisms inherited from Mukai's model can be 
generated by the order 6 automorphism
\[
\varphi_1: \; (t,x,y) \mapsto (x+1,t+1,-1/y)
\]
(such that $\varphi_0=\varphi_1^3$) and the involution 
\[
\varphi_2: \; (t,x,y) \mapsto (-t,-x,-y).
\]
It follows that $G$ can be obtained from $G_0$ (acting on both $X_\CC$ and $\tilde X$) 
by adding the wild automorphism 
$$
\varphi = \varphi_3: \; (t,x,y) \mapsto (t,x+1,y)
$$
of the generic fibre of $\pi$
to which we already alluded above.

%
%
%We shall now show that $G\cong 2^4:\mathfrak S_{3,3}$ as desired.
%
%Note that the normal subgroup $N=2^4\lhd G_0$
%is generated by sign changes in the homogeneous coordinates of \eqref{eq:X3}.
%From the action on the nodes, one infers that $N$ corresponds to compositions of the translations 
%by 2-torsion sections of the two fibrations $\pi, \pi'$.
%In fact, one can show that $G\hookrightarrow \mathfrak S_8$, corresponding to the nodes,
%or also $G\hookrightarrow O(L)$).

\begin{claim}
$N\lhd G$.
\end{claim}

\begin{proof}
Since $N\lhd G_0$, it suffices to prove
that $\varphi N=N\varphi$
 for the remaining generator $\varphi$ of $G$. 
 Since $\varphi$ preserves the fibration $\pi$ over $\PP^1_t$ fibrewise,
i.e.\ it acts trivially on the corresponding nodes,
the claim already follows from the standard fact that $V_4\lhd \mathfrak S_4$
for the Klein 4-group $V_4$ generated by the translations of the two-torsion sections of $\pi$.
\end{proof}

To complete the proof of Proposition \ref{prop:121}, it remains to exhibit a subgroup 
$\mathfrak S_{3,3}\subset\Aut^{\text{symp}}(\tilde X)$
such that
\begin{eqnarray}
\label{eq:S33}
G/N  \cong \mathfrak S_{3,3} \cong \mathfrak S_3\times\mathfrak S_3.
\end{eqnarray}
For this purpose, we consider automorphisms
which act linearly on $x,t$ in the affine model \eqref{eq:Km}
 (so that they preserve the zero sections of both fibrations $\pi$ and $\pi'$,
as suggested by the quotient $G/N$).
In detail, we augment the subgroup 
$\mathfrak S_3\cong \langle \varphi_1^2, \varphi_2\circ\varphi_0\rangle\subset D_{12}$ 
by the commuting copy given by
\[
\mathfrak S_3 \cong\langle \varphi_1\circ\varphi_3, \varphi_0\rangle
\]
and verify explicitly by computing all products that  \eqref{eq:S33} holds true.
This completes the proof of Proposition \ref{prop:121}.
\end{proof}

\subsubsection{$\boldsymbol{G=2^2.\mathfrak A_{4,4}}$}
\label{sss:112}

The group $G=2^2.\mathfrak A_{4,4}$ contains $G_0 = H_{192}$ as an index 6 subgroup,
so we start again with Mukai's model $\tilde X$ arising from \eqref{eq:X3}.
We will argue with the fixed lattice $\NS(\tilde X)^{G_0}$
which clearly contains the classes of
\begin{eqnarray}
\label{eq:inv_div}
\begin{cases}
D_1 =  \frac 12\sum(\text{exceptional curves}),\\
\\
 D_2 = \frac 12 \sum(\text{strict transforms of the lines}).
 \end{cases}
 \end{eqnarray}
Together, they generate a rank two lattice $L_0\subset \NS(\tilde X)^{G_0}$ which in fact equals 
$$
L_0 = \begin{pmatrix}
-4 & 8\\ 8 & -8
\end{pmatrix}
\cong\langle 4,-8 \rangle
\cong
\NS(\tilde X)^{G'} 
$$
for $G' =  2^4:\mathfrak S_{3,3}$ from the previous section (in agreement with \cite[Table 1]{HM290}).
For $\NS(\tilde X)^{G_0}$, we need a further invariant generator which can be derived as follows.

Start with a certain conic $Q\subset X$;
on $\tilde X$, this induces a section for both fibrations $\pi, \pi'$.
In terms of $\pi$, with Weierstrass form \eqref{eq:Km}, 
we can write $Q = (t,1)$ and compute that
the $G_0$-orbit of $Q$ comprises the eight sections
$\pm Q + R$,
where $R$ runs through the 2-torsion sections (including the zero section $O$)
and the signs indicate, here and below, opposite sections with respect to the group structure.
By \cite{MWL}, their orbit sum is zero modulo $\mbox{Triv}(\pi)$,
the lattice generated by fibre components and zero section;
 in particular, the orbit sum lies in $L_0$ for symmetry reasons,
 so $Q$ does not contribute the missing generator for $\NS(\tilde X)^{G_0}$.

\begin{lemma}
\label{claim:MWL}
A $\ZZ$-basis of $\MWL(\pi)$
is given by the four sections $Q$,
$$
\varphi^* Q = (t-1,1), \;\; \iota^* Q=(-t,i)  \;\; \text{ and } \;\; \varphi^* \iota^* Q = (-t-1,i) \;\;\; (i^2=-1)
$$
where $\iota: (x,y) \mapsto (-x,iy)$ denotes an order 4 automorphism of the generic fibre $E$.
\end{lemma}

\begin{proof}
Each section meets the non-identity fibre components at $t=0,1,-1$ as indicated by the $x$-coordinate,
evaluated at the respective $t$.
In terms of the height pairing of the Mordell--Weil lattice  \cite{MWL},
it follows that the sections have height $1$ and that
\[
\langle Q,\varphi^* Q\rangle = \langle \iota^*Q,\varphi^*\iota^*Q\rangle = -1/2
\]
(because for both pairs, the two sections meet transversally on the fibre at $t=\infty$),
the other pairings being zero.
Hence we get an inclusion of the sublattice
\[
\MWL(\pi) \supseteq \langle Q,\varphi^* Q\rangle \oplus \langle \iota^* Q, \varphi^*\iota^* Q\rangle \cong A_2[1/2]^2.
\]
Then the statement about the $\ZZ$-basis follows from the determinant formula \cite[Cor.\ 6.39]{SS}
by  comparing the determinants $-16$ of $L$ and $9/16$ of the above
sublattice of $\MWL(\pi)$ to $\det\NS(\tilde X)=-9$.
\end{proof}

\begin{remark}
It follows that $\MW(\pi)$ is generated by the above 4 sections and the 2-torsion sections,
and analogously for $\MW(\pi')$.
\end{remark}

For later use, we compute the missing generator of the rank 3 fixed lattice $\NS(\tilde X)^{G_0}$.
To this end, we consider the action of  the  generators $\varphi_0, \varphi_1^2$ and $\varphi_2$ of $D_{12}$
on $\MW(\pi)$ to find that the divisor
\[
D_3 = \frac 12 (\varphi^* Q + (-(\varphi^*)^2 Q) + (\text{translates by torsion sections}))
\]
is invariant under $G_0$.

\begin{claim}
\label{claim:inv_G_)}
$\NS(\tilde X)^{G_0} = \left\langle D_1, D_2, D_3\right\rangle$.
\end{claim}

\begin{proof}
Since each given divisor is $G_0$-invariant,
it suffices to compute their Gram matrix
to verify that they do not only generate a full rank sublattice of $\NS(\tilde X)^{G_0}$,
but, by inspection of its determinant $-384$ which agrees with that of $\Lambda^G$ up to sign by \cite[Table 1]{HM290},
the whole invariant lattice
$\NS(\tilde X)^{G_0}$.
\end{proof}

The main idea to realize a symplectic action of $G$ on $\tilde X$ 
is to combine the action of $G_0$ with that of another subgroup of $G$, namely  $\tilde G = T_{48}$ in the notation of %\cite{HM290}, 
\cite{Mukai}.
Over $\CC$, this group acts on the smooth double sextic given by
\[
Y: \;\;\; w^2 = xy(x^4+y^4)+z^6
\]
where the group action essentially is given by GL$(2,\FF_3)$ acting on the homogeneous $x,y$ coordinates.
It will be useful to note that, in any characteristic $\neq 2$, the group $T_{48}$ contains an element of order $8$, defined affinely by
\[
\psi: \; (x,y,w) \mapsto (\zeta_8 x, \zeta^3_8 y, -w)
\]
for a primitive 8th roots of unity $\zeta_8$.
Especially, characteristic $3$ leads to 8 singular points 
$P_1,\hdots,P_8
$, 
two  at each root of $x^4-y^4$.
All singular points have type $A_2$, forming a single $\tilde G$-orbit.
On the minimal resolution $\tilde Y$, the exceptional curves $E_i^{\pm}$ above the $P_i$ form two $\tilde G$-orbits distinguished by the sign.

In addition, $Y$ contains 14 pairs of lines (i.e. pairs of smooth rational curves on $Y$ which map to a line in $\PP^2$).
The first 8 pairs, $\ell_i^\pm$, form again two orbits, distinguished by the signs
 starting from the lines $\ell_1^\pm$ given by 
 $$x+y+z=0,
 $$
 say.
 On $\tilde Y$, each line $\ell_i^\bullet$ meets exactly three $E_j^\bullet$ (with the same sign),
 while being disjoint to the line with the opposite sign which also meets the opposite exceptional curves.
Moreover, $\ell_j^\bullet$ meets exactly one other line, with opposite sign.
The other 6 pairs of lines, given by the zeroes of the homogenous form $xy(x^4+y^4)$, will not be relevant to us,
so we only depict the first named 28 smooth rational curves on $\tilde Y$ in Figure \ref{fig}.
Note that, for ease of presentation, the 20 future fibre components (printed in black) appear only once
while the 12 future sections (printed in green) are indicated at each fibre component which they meet.

\begin{figure}[ht!]
%\centering
\includegraphics[width=0.9\textwidth]{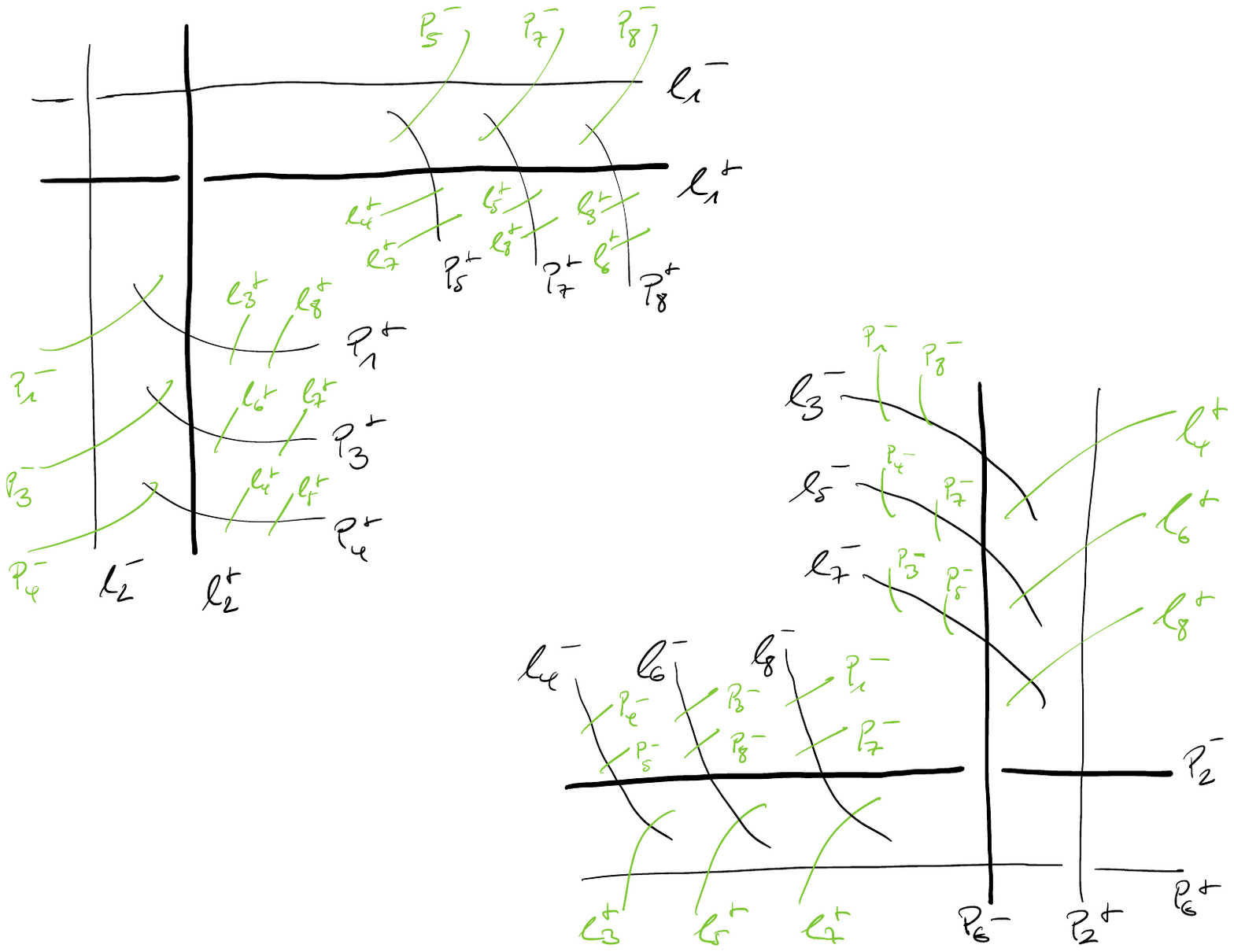}
\caption{32 smooth rational curves $E_i^\pm, \ell_j^\pm$ on $\tilde Y$}
\label{fig}
\end{figure}

\begin{lemma}
\label{lem:Y}
\begin{enumerate}
\item[(i)]
$\tilde Y$ is supersingular of Artin invariant $\sigma_0=1$.
\item[(ii)]
$\NS(\tilde Y)^{\tilde G} \cong \langle 2\rangle \oplus A_2[8]$, generated 
by the pull-back of a hyperplane section and the sums $\sum E_j^\bullet$.
\end{enumerate}
\end{lemma}

\begin{proof}
(i)
Since $Y_\CC$ has transcendental lattice $A_2[-8]$ by \cite{Hashi},
this follows as in the proof of Claim \ref{claim:ss}.
(Alternatively, one can verify that the Gram matrix of the given 44 smooth rational curves has rank $22$,
and exhibit a maximal minor of determinant $-9$.)

(ii)
Computing $\tilde G$-invariants, consider the hyperplane section
and the sums $\sum E_i^\bullet$ as stated.
This gives $\NS(\tilde Y)^{\tilde G} \cong \langle 2\rangle \oplus A_2[8]$,
in agreement with \cite[Table 1]{HM290}.
(It will be crucial in what follows that each sum $\sum \ell_i^\bullet$ is $\tilde G$-invariant, too.)
\end{proof}

To connect between the two groups $G_0=H_{192}$ and $\tilde G = T_{48}$,
we work with the isotrivial elliptic fibration \eqref{eq:Km} which we derived from $\tilde X$ in \ref{sss:121}.
On $\tilde Y$, this can be exhibited by identifying the following disjoint divisors of Kodaira type I$_0^*$:
\begin{eqnarray*}
\tilde D_{-1} = E_1^++E_3^++E_4^++\ell_1^-+2\ell_2^+, & & 
\tilde D_\infty = E_5^++E_7^++E_8^++\ell_2^-+2\ell_1^+,\\
\tilde D_1 = E_2^+ +\ell_4^-+\ell_6^-+\ell_8^-+2E_2^- ,&&
\tilde D_0 = E_6^++\ell_3^-+\ell_5^-+\ell_7^-+2E_6^-.
\end{eqnarray*}
This has twelve disjoint sections given by the remaining six exceptional curves $E_i^- (i\neq 2,6)$ 
and six lines $\ell_j^+ (j>2)$.
It follows from Lemma \ref{lem:Y} (i) that the fibration can be converted to the shape of \eqref{eq:Km}
because 
the fibration is automatically isotrivial, leading to Km$(E\times E')$ as before.
This Kummer surface is supersingular if and only if $E$ and $E'$ are,
but there is only one supersingular elliptic curve in characteristic $3$, namely $E_0\cong E\cong E'$.
We then use M\"obius transformations to locate each fibre at the point of $\PP^1$ indicated by the index.
After choosing $P_1^-$ as zero section, there is still enough freedom, using the automorphisms $\varphi_3$ and $\iota$, to arrange
for 
$$
P_5^- = P_{-1}, \;\;  \ell_5^+ = P_0, \;\; \ell_6^+ = P_1.
$$
For  the remaining eight sections,
 we can read off from the fibre components intersected that,
 after composing with the involution of the generic fibre, if necessary,
 \[
 P_3^- = \varphi^* Q \;\;\; \text{ and } \;\;\; P_4^- = -(\varphi^*)^2 Q
 \]
 where the sign follows from the property that $P_3^-, P_4^-$ are disjoint
 (cf.\ the proof of Lemma \ref{claim:MWL}).
 
 \begin{lemma}
 \label{claim:hat_G}
 Under these identifications, let $\hat G=\langle G_0, \tilde G\rangle$. Then
$$
\NS(\tilde X)^{\hat G} = \left\langle \frac 12\sum(P_i^++\ell_i^-),\, \frac 12 \sum(P_i^-+\ell_i^+)\right\rangle
\cong \begin{pmatrix}
-8 & 16\\
16 & -8
\end{pmatrix}.
$$
 \end{lemma}

\begin{proof}
This follows directly from our descriptions of $\NS(\tilde X)^{G_0}$ and $\NS(\tilde Y)^{\tilde G}$,
since the given curves exactly give  fibre components, torsion sections and the sections supporting
the effective divisor $2D_3$.
\end{proof}

\begin{corollary}
Under the above identifications, $\hat G$ is finite.
\end{corollary}

\begin{proof}
Since $\NS(\tilde X)^{\hat G} $ is hyperbolic, it follows that $\hat G$ preserves a positive class.
Hence it acts faithfully on some negative definite lattice and thus is finite.
\end{proof}

\begin{remark}
Other identifications lead to $\NS(\tilde X)^{\hat G}$ of smaller rank, possibly even empty,
where thus $\hat G$ is forced to be infinite.
\end{remark}

\begin{proposition}
Under the above identifications, $\hat G = G = 2^2.\mathfrak A_{4,4}$.
\end{proposition}

\begin{proof}
By Lemma  \ref{claim:hat_G}, the determinant of $\NS(\tilde X)^{\hat G}$ lies in the square class of $-3$.
Among the admissible  subgroups of $O(\Lambda)$ with fixed lattice $\Lambda^G$ of rank $4$,
as classified in the previous sections,
there is only one  maximal group other than $G$ with $\det(\Lambda^G)=3 \mod(\QQ^*)^2$
in \cite[Table 1]{HM290}, namely $3^{1+4}:2.2^2$.
Since its order 1944 is not divisible by $\# G_0=192$, we infer that $\hat G$ can only be $G$ or a subgroup thereof.
Clearly, $\hat G\neq G_0$ since  $\hat G$ contains an element of order $8$ (namely $\psi\in\tilde G$)
while $G_0$ doesn't.
Assuming that $\hat G\subsetneq G$, it follows that
$\hat G$ is a subgroup of index $2$ (since it contains $G_0$).
Since $G$ and also $\hat G$ are subgroups of $M_{23}$ acting with 4 orbits,
these subgroups are known, and indeed $\hat G = 2^2.(\mathfrak A_4\times\mathfrak A_4)$.
But this group does not contain an element of order 8, either.
This contradiction proves that $\hat G=G$.
\end{proof}

\begin{remark}
The above argument provides a recipe
to exhibit actions of all maximal groups from Theorems \ref{thm}, \ref{thm1} explicitly  in a given characteristic.
Indeed, one can not only use Mukai's models from \cite{Mukai}
to get isomorphic supersingular reductions, but also those from \cite{BS} or \cite{BH}, for instance.
\end{remark}

\subsection{Explicit realizations of wild groups in rank 4}
\label{ss:explicit-rk4}

To supplement Table \ref{table6}, we note that the abstract realizations in the given characteristics can also be exhibited explicitly.

\subsubsection{Characteristic $11$}
For %the groups with an element of order $11$,
$L_2(11)$,
an explicit realization 
on the supersingular K3 surface of Artin invariant $\sigma=1$ in characteristic $11$ is given in \cite{DKeum}.

\subsubsection{Characteristic $7$}
\label{sss:7}

For the group $L_2(7)\times 2$, % and its subgroups $2\times(7:3)$ and $14$,
one can consider the elliptic fibration
\[
y^2 = x^3 + (t^8+1) x
\]
as in \cite{DKeum}.
In characteristic $7$, the K3 surface (supersingular of Artin invariant $1$)
inherits a  symplectic action of the simple group PSU$(2, \FF_{49})\cong L_2(7)$ induced from the base of the fibration.
The group is in fact to be found in Mukai's list of maximal groups with 5 orbits,
but presently the group action can be extended by translation by the 2-torsion section $(0,0)$
to yield a symplectic group action by $L_2(7)\times 2$ (and its subgroups).

\subsubsection{Characteristic $5$}
\label{sss:5}

Consider the K3 surface given by 
\[
y^2  + (t^6+1)y = x^3.
\]
%Applying the degree 6 base change $t=s^6+1$, we obtain a 1-dimensional family of K3 surfaces.
In characteristic $5$, this comes with a natural symplectic action by the simple group $L_2(5)\cong \mathfrak A_5$
acting on the base $\PP^1$ of the fibration.
Together with the translations by the 3-torsion sections,
this realizes $3 \times  \mathfrak A_5$.
Note that this group admits a degree 2 extension given by composing the involution of the fibres
with the involution of the base $(t\mapsto -t)$;
this yields  the maximal subgroup $\mathfrak A_{3,5} = (3\times \mathfrak A_5):2\subset M_{23}$ from Table \ref{table6}.

\section{Proofs of the main theorems}
\label{s:proofs}

We conclude the paper by collecting the proofs of the theorems from the introduction.

\subsection{Proof of Theorem \ref{thm}}

For K3 surfaces of finite height, 
Theorem \ref{thm} was proved in \cite[Thm.\ 4.7]{DKeum}
(to the extent that, just like in Theorem \ref{thm:Mukai} over $\CC$, only subgroups of $M_{23}$
with 5 orbits may occur).

For supersingular K3 surfaces,
Theorem \ref{thm} follows from  the combination of Theorem \ref{thm:tame},
for groups with orbit length 4, also proving that they can only be realized on $X_{0,p}$,
and Theorem \ref{thm:equiv} which also settles the existence of each group action on $X_{0,p}$.

%and Proposition \ref{prop:Mukai}.
%In a little detail, we apply the reduction to Artin invariant 1 by Proposition \ref{prop:reduction} 
%and use the following observation:
%
%\begin{obs}
%Any group $G$ which can be realized symplectically
%on $X_{0,p}$ in characteristic $p$,
%but does occur as subgroup of the groups in Theorem \ref{thm:tame} (Table \ref{tab})
%and Proposition \ref{prop:Mukai} (Table \ref{table7}),
%has $p\mid\det\Lambda^G$, so $p\mid |G|$ by Lemma \ref{lem:aux}
%and the group cannot be tame.
%\qed
%\end{obs}
%

\subsection{Proof of Theorem \ref{thm1}}

Theorem \ref{thm1} is a consequence of all the groups from \cite{HM290} 
being excluded or realized (sometimes in speciific characteristics) in the course of this paper
(using the crystalline Torelli Theorem \ref{thm:cTorelli} and the explicit geometric models in Section \ref{s:models},
whenever there is no crystalline Torelli theorem available).
Indeed, 
 when reaching Conclusion \ref{conclusion:all}, 
 we checked the hierarchy of subgroups  continuously, 
 using \cite{Atlas} and the auxiliary files of \cite{HM290},
 to confirm the maximality statement of Theorem \ref{thm1}.
\qed

%
%\subsection{Proof of Corollary \ref{cor}}
%
%By \cite{DKeum}, wild finite group actions may only occur in characteristic $p\leq 11$.
%Hence the corollary follows directly from Theorem \ref{thm1}.
%\qed

\subsection{Proof of Corollary \ref{cor:11}}

By \cite{DK-11}, a wild group $G$ acts symplectically on some K3 surface in characteristic $11$
if and only if $G\subset M_{11}$ or $M_{22}$, in agreement with Corollary \ref{cor}.
(\cite[p.\ 800]{DK-11} states $L_2(11)$ as third maximal group,
but that group is contained in both, $M_{11}$ and $M_{22}$.)

The possible tame groups in characteristic $11$ are classified in Theorem \ref{thm}.
By Theorem \ref{thm1} and its proof,
each tame group is contained in  $M_{11}, M_{22}$ or $2^4:\mathfrak S_{3,3}$
as stated.
\qed

\begin{remark}
It also follows that each finite symplectic group $G$ acting symplectically on some K3 surface in characteristic $11$
admits a symplectic action on $X_{0,11}$.
\end{remark}

\subsection{Proof of Theorem \ref{thm:cohom}}

Theorem \ref{thm:cohom} is an immediate combination of Theorem \ref{thm} and Theorem \ref{thm1},
since the  groups listed separately in Theorem \ref{thm:cohom} (2) -- (4) are exactly the groups
from Theorem \ref{thm1} which are not contained in $M_{23}$.
\qed

\subsection{Proof of Theorem \ref{thm:holds}}

To prove Theorem \ref{thm:holds}, we note that
characteristics $p=2,3,5,7$ are excluded by the condition that $p$ be a non-zero square,
and $p=11$ by the square condition at $2$ or $7$.
Thus we reduce to the case $p=$ char$(k)>11$
whence  it follows from \cite{DKeum} that any finite symplectic group is automatically tame,
and we can apply Theorem \ref{thm}.

To see that the square condition is sufficient, we note that any group $G$ from Table \ref{tab}
has $\det(\Lambda_G)$ (or orbit lengths) a product of  primes $p'\leq 7$.
If these are all squares modulo $p$, then \eqref{eq:Legendre} obviously cannot return $-1$.

To see that the square condition is necessary, 
we simply note that for each prime $p'\leq 7$, there is a group $G'$ in Table \ref{tab}
such that $\det\Lambda_{G'}=p'$ modulo squares (or equivalently, the product of orbit lengths 
is $p'$ times a square, cf.\ Theorem \ref{thm:tame}).
Therefore, if $p'$ were a non-square modulo $p$,
then $G'$ could be realized on $X_{0,p}$ by Theorem \ref{thm}.

Spelling out the square condition, we arrive at the equivalent characterization of the  primes $p$
where Mukai's classification holds true.
\qed

\subsection*{Acknowledgement}
We are grateful to the referees for the thorough comments which greatly helped us improve the paper.
We thank S.\ Brandhorst, I.\ Dolgachev and S.\ Muller for very useful comments.

\end{document}